\documentclass[12pt]{elsarticle}

\usepackage{amsmath, amssymb}
\usepackage{graphics}
\usepackage{caption}
\usepackage{amsthm}
\usepackage{algorithm}
\usepackage{algorithmic}
\usepackage{cancel}
\usepackage{caption}
\usepackage{subcaption}

\def\deg{{\rm deg}}

\newtheorem{theorem}{{\bf Theorem}}
\newtheorem{remark}{{\bf Remark}}
\newtheorem{definition}[theorem]{{\bf Definition}}
\newtheorem{corollary}[theorem]{{\bf Corollary}}
\newtheorem{proposition}[theorem]{{\bf Proposition}}
\newtheorem{lemma}[theorem]{{\bf Lemma}}
\newtheorem{example}{{\bf Example}}

\def\bfx{\boldsymbol{x}}

\def\OOO{\mathcal{O}}
\def\tOOO{\tilde{\mathcal{O}}}

\begin{document}

 \begin{frontmatter}
\title{A new method to compute the singularities of offsets to rational plane curves.}


\author[a]{Juan Gerardo Alc\'azar\fnref{proy,proy2}}
\ead{juange.alcazar@uah.es}
\author[b]{Jorge Caravantes\fnref{proy}}
\ead{jcaravan@mat.ucm.es}
\author[c]{ Gema M. Diaz-Toca\fnref{proy,proy3}}
\ead{gemadiaz@um.es}

\address[a]{Departamento de F\'{\i}sica y Matem\'aticas, Universidad de Alcal\'a,
E-28871 Madrid, Spain}
\address[b]{Departamento de \'Algebra, Universidad Complutense de Madrid, E-28040 Madrid, Spain}
\address[c]{Departamento de Matem\'atica Aplicada, Universidad de Murcia,  E-30100 Murcia, Spain}

 \fntext[proy]{
 Supported by the Spanish ``Ministerio de
 Econom\'ia y Competitividad" under the Project MTM2014-54141-P. }
 
 \fntext[proy2] {Member of the Research Group {\sc asynacs} (Ref. {\sc ccee2011/r34}) }

 \fntext[proy3]{Supported by the Research Group {\sc E078-04} of the University of Murcia}


\begin{abstract}
Given a planar curve defined by means of a real rational parametrization, we prove that the affine values of the parameter generating the real singularities of the offset are real roots of a univariate polynomial that can be derived from the parametrization of the original curve, without computing or making use of the implicit equation of the offset. By using this result, a finite set containing all the real singularities of the offset, and in particular all the real self-intersections of the offset, can be computed. We also report on experiments carried out in the computer algebra system Maple, showing the efficiency of the algorithm for moderate degrees. 
\end{abstract}

  \end{frontmatter}

\section{Introduction}\label{section-introduction}

Intuitively, the \emph{offset curves} to a given curve are ``parallel" curves to the original curve, called the \emph{generator} curve, at a certain distance. The offsetting operation is important in Computer Aided Geometric Design (CAGD), because it can be used to give ``thickness" to an object, in this case a curve, which is thin; the offsetting distance can be regarded as the desired ``thickness" of the new object. Furthermore, offsets also have applications in fields like robotics or manufacturing \cite{Farouki}, \cite{Patrikalakkis}. 

When we compute the offset of a curve, typically we want to reproduce in the offset the shape of the generator. However, sometimes the offsetting operation introduces singularities that destroy the topology of the original curve. Hence, one has to identify the parts of the offset that should be discarded, and trim them away in a post-processing step. For example, in Figure \ref{fig:parab} (left) we can see the parabola $y=x^2$ (in thin line) and its offset at distance $d=1$ (in thick line). We observe that while one of the connected components of the offset has the topology of the parabola, the other component has a different topology. This last component has one self-intersection and two cusps. After trimming away the loop containing the three singularities (see Figure \ref{fig:parab}, right), the two components of the new curve that we get, called the \emph{trimmed offset}, have the topology of the original curve. Furthermore, in Section 4 of \cite{Far1} it is proven that the trimmed offset can be easily computed whenever the parameter values of the self-intersections of the offset, or even a finite set containing them, are known. So the computation of the self-intersections of the offset is strongly related with the trimming operation.

Trimming is important in computer aided design to keep the original shape and therefore improve the appearance of the image, but also in manufacturing. More precisely, if a certain curve is to be machined by the cylindrical cutter of a milling machine, the cutter follows a certain-line trajectory specified by the offset, where the offsetting distance equals the cutter radius (see page 162 of \cite{Farouki}). However, if the offset has self-intersections then the offset will have loops (like the one in the offset to the parabola, see Figure \ref{fig:parab}, left), giving rise to the problem of ``gouging", as it is called in NC machining (see Section 11 in \cite{Patrikalakkis}). Essentially, gouging implies that these small loops must be removed later by using a smaller size cutter.

\begin{figure}
\begin{center}
$\begin{array}{cc}
\includegraphics[scale=0.3]{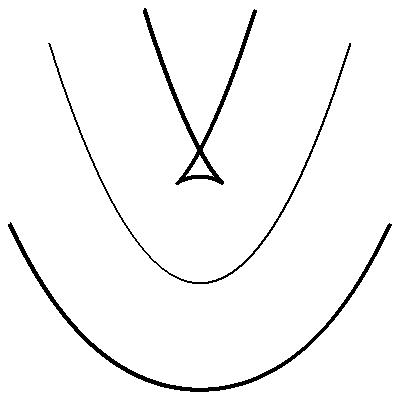} & \includegraphics[scale=0.3]{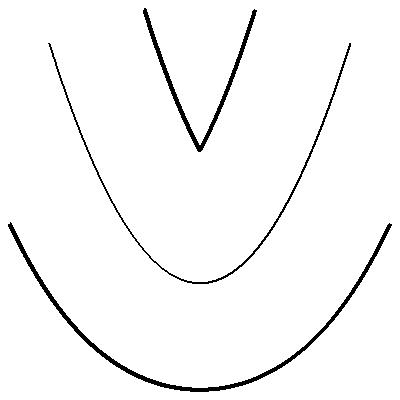}
\end{array}$
\end{center}
\caption{Offset (left), and trimmed offset (right) of the parabola at the distance $d=1$.}\label{fig:parab}
\end{figure}

 In this paper we deal with the problem of computing the singularities of the offset to a generator curve given by means of a rational parametrization, as it is common in CAGD. More precisely, we are interested in computing the affine values of the parameter giving rise to the offset singularities. In this context, a first difficulty is the fact that the offset does not need to be rational; in fact, if the offset is rational then the computation of its singularities is relatively easy, and can be done for instance by using the method in \cite{PD07}. Since the offset to an algebraic curve is also algebraic, one might try to overcome the aforementioned difficulty by working with the implicit equation of the offset, in order to derive the singularities of the offset from there. Nevertheless, the offsetting process causes kind of an ``explosion", so that the implicit equation of the offset is much more complicated than the original one, even for very simple curves. Therefore, deriving the singularities of the offset from its implicit equation is very time-consuming and quite often impossible in practice. 

As a consequence, many papers in the literature have proposed methods to derive the singularities of the offset by computing them directly from the generator curve, or by approximating the offset by another object which is easier to manipulate. In this sense, one should distinguish between ``local" singularities, i.e. singularities, like for instance cusps, that are due to local phenomena, and self-intersections, which are due to the intersection of different branches of the offset. Local singularities of offsets to rational curves are easy to find, since the values of the parameter $t$ generating them are the solutions of the equation $k(t)=-1/d$, where $d$ is the offsetting distance and $k(t)$ is the curvature of the curve. For singularities of this type coming from regular points of the generator, one can see Section 2.5 of \cite{Far1}. In fact, the results in \cite{Far1} are applicable not only to rational curves, but to parametric curves defined by means of a {\it regular} parametrization, i.e. a parametrization where the speed vector does not vanish. For the analysis of local offset singularities coming from singular points of the generator of an algebraic curve, one can see \cite{AS07}. However, finding the self-intersections of the offset is a difficult problem. 

In \cite{Far2} and \cite{Maekawa}, the self-intersections are directly derived from the generator. In \cite{Far2}, this is done for the case where the initial curve admits a polynomial parametrization. In order to do this, the parameter values generating the self-intersections are proven to be the roots of a polynomial which is the quotient of a big determinant and a product of two polynomials. However, the generalization to the case when the generator admits a rational, non-polynomial, parametrization is cumbersome. 

In \cite{Maekawa}, the self-intersections of the offset are computed by solving a system in four variables and four equations, also derived from the generator curve. The idea is also applicable to a general regular parametrization, not necessarily rational. Nevertheless, in that case one has to deal with the numerical problem of approaching all the solutions of a nonlinear system. Notice that if a non-rational parametrization is used, the system is not necessarily algebraic. However, if the parametrization of the original curve is polynomial, the system is certainly algebraic. In this case, by writing the parametrization of the original form in Bezier form, in turn one can write the equations of the system as Bernstein polynomial equations. Then the solution of the system is reduced to finding the intersection of two bivariate Bezier patches with a certain plane. In order to do this, de Casteljau subdivision methods coupled with rounded interval arithmetic are involved. The generalization to rational curves, though stated to be feasible, is mentioned as a topic of future research. 

A second possibility, that has been explored by many authors, is to approximate the offset by means of a simpler object, and then approximate the self-intersections, and therefore the trimmed offset itself, from that object. In some cases the offset is approximated by a rational curve \cite{Seong}, or a polynomially parametrized curve \cite{PEK}. In other cases \cite{Chiang}, \cite{Lee}, \cite{OF} an approximation with a polygonal line is used. In \cite{Kim}, the input is a planar rational curve, and a $G_1$-continuous biarc approximation of the curve is employed. Some other approaches to the problem and additional references can be found in Section 11.2.4 of \cite{Mae2}. 

In this paper, we provide a new method to find a finite set containing all the affine parameter values giving rise to real, non-isolated singularities of the offset. Our method computes these values from the generator curve, and does not require to compute or make use of the implicit equation of the offset. We were inspired by the ideas of \cite{Fukushima}, which in turn is related with \cite{AS99}, where the computation of the genus of the offset from the genus of the generator curve is addressed. The main idea of the method is the following: by \cite{Fukushima}, one can establish a birational mapping between the offset and a much simpler curve. However, at the self-intersections of the offset, this mapping cannot be inverted. Additionally, one can prove that the same holds not only for the self-intersections of the offset, but for all real singularities of the offset. From a computational point of view, our algorithm uses \emph{subresultants} as an essential tool, jointly with root finding. Furthermore, we have implemented and tested our algorithm in the computer algebra system Maple; the code can be freely downloaded from \cite{gmdt}. 

Compared to other methods, our algorithm does not use any approximation of the offset, and can be applied to possibly singular, rational curves. The algorithm has at least two advantages: first, the implementation is easy and requires only a few lines of code. Second, the description is, unlike \cite{Far2}, \cite{Maekawa}, basically the same regardless of whether the parametrization is polynomial or rational. In fact, the algorithm is presented under the assumption that the parametrization is non-polynomial. As a disadvantage, we can mention the potential appearance in the output of superfluous values of the parameter. This does not happen when the parametrization is polynomial, but it can happen, in certain cases, when it is non-polynomial. Nevertheless, even in this case our results can be applied to the offset trimming problem, since the presence of superfluous values does not affect the final result. We must also observe that when 
approached from a symbolic or symbolic-numeric point of view, the problem is inherently difficult, because the degree of the offset of a curve of degree $n$ is bounded by $2(3n-2)$ (see Theorem 3.6 in \cite{Far2}), and the growing of the coefficients can be serious. For instance, the offset of the Descartes' Folium $x^3+y^3-3xy=0$ generically has degree 14 and 114 terms. In spite of this fact, our experiments show a good performance of our algorithm for moderate, but far from trivial, examples. These examples include some curves analyzed in \cite{Far2}, \cite{Maekawa}, in order to compare with those methods.

The structure of the paper is the following. Generalities on offsets and subresultants are provided in Section \ref{sec-prelim}. The strategy behind our method is presented in Section \ref{sec-inv}. This strategy is at first aimed to compute the self-intersections of the curve; however, we prove that, as a by-product, we also get the remaining singularities of the offset. The main result of the paper, jointly with the algorithm it gives rise to, are given in Section \ref{final-th}; details on examples, as well as a thorough analysis of the complexity of the algorithm and the growing of the coefficients, are also provided here. The conclusions of the paper are presented in Section \ref{sec-conclusions}. Although our algorithm can be easily described and implemented, the proof of the result it is based on takes certain work; the parts of the proof that are not essential to understand the main result of the paper and the subsequent algorithm, are given in Appendix I, Appendix II and Appendix III.

\section{Preliminaries and generalities.} \label{sec-prelim}

\subsection{The offset curve} \label{subsec-off}
\noindent Let ${\mathcal C}$ be a real, rational plane curve, parametrized by
\begin{equation} \label{curve}
\phi(t)=\left(\mathcal{X}(t),\mathcal{Y}(t)\right)=\left(\frac{X(t)}{W(t)},\frac{Y(t)}{W(t)}\right),
\end{equation}
where $X(t),Y(t),W(t)$ are polynomials with real coefficients, and \[\gcd(X(t),Y(t),W(t))=1.\] We also assume that $\phi(t)$ is {\it proper}, i.e. birational or equivalently injective except perhaps for finitely many values of the parameter $t$. This condition ensures that $\mathcal{C}$ is a reduced curve (see Theorem 4.41 in \cite{SWPD}). Note that properness can always be achieved by reparametrizing the curve, if necessary \cite{SWPD}. 

The \emph{offset} to ${\mathcal C}$ at distance $d\in {\Bbb R}^+$, ${\mathcal O}_d({\mathcal C})$, is defined as the \emph{Zariski closure} of the set of points $(x,y)=\phi_d(t)$, where 
\begin{equation} \label{offset}
\phi_d(t)=\left(\frac{X(t)}{W(t)}\pm d \frac{V(t)}{\sqrt{U^2(t)+V^2(t)}},\frac{Y(t)}{W(t)}\mp d\frac{U(t)}{\sqrt{U^2(t)+V^2(t)}}\right),
\end{equation}
with
\begin{equation}\label{UV}
U(t)=X'(t)W(t)-X(t)W'(t),\mbox{ }V(t)=Y'(t)W(t)-Y(t)W'(t).
\end{equation} 
Furthermore, in the paper we will assume that $t$ is a real value. Therefore, $(x,y)=\phi_d(t)$ means that the Euclidean distance between $(x,y)\in {\mathcal O}_d({\mathcal C})$ and the point $p=\phi(t)\in {\mathcal C}$, measured along the normal line to ${\mathcal C}$ through $p=\phi(t)$, is $d$; we say then that $p=\phi(t)$ {\it generates} $(x,y)$. When the first sign of $\pm$ and $\mp$ in the expression \eqref{offset} is considered, the geometrical locus described is called the \emph{exterior offset}; if the second sign is chosen, the geometrical locus described this way is called the \emph{interior offset}. Furthermore, $\phi_d(t)$ can be extended to the $t_0$ values where $U^2(t_0)+V^2(t_0)=0$ by just taking limits $t\to t_0$ (in the usual topology); the points computed this way also belong to the offset. Additionally, if $\mbox{lim}_{t\to \infty}\phi(t)$ is an affine point, which happens iff $\deg(X(t))\leq \deg(W(t))$ and $\deg(Y(t))\leq \deg(W(t))$, $\mbox{lim}_{t\to \infty}\phi_d(t)$ generates two more points, which we will denote as $P_{\pm \infty}$, also belonging to the offset. 

The computation of an implicit equation $F(x,y)=0$ of ${\mathcal O}_d({\mathcal C})$ is addressed in \cite{Far2}. Let us review some of the ideas of \cite{Far2}, which are relevant for our purposes. In \cite{Far2} the following polynomials are introduced:
\begin{eqnarray} 
\tilde{P}(x,y,t):= U(t)( W(t)x-X(t) ) +V(t)( W(t)y-Y(t) ) = 0, \label{P}\\
\tilde{ Q}(x,y,t):= ( W(t) x - X(t) )^2 + ( W(t)y - Y(t) )^2 - d^2 W^2(t) = 0.\label{Q}  
\end{eqnarray}
%
For the values of $t$ which satisfy 
$W(t)\neq0$, $(U(t),V(t))\neq (0,0),$ 
the equation \eqref{P} represents the normal line to ${\mathcal C}$ at the point $\left(\frac{X(t)}{W(t)},\frac{Y(t)}{W(t)}\right)$, while the equation \eqref{Q} represents the circle of radius $d$ centered at the point $\left(\frac{X(t)}{W(t)},\frac{Y(t)}{W(t)}\right)$. The implicit equation  of ${\mathcal O}_d({\mathcal C})$ is determined by eliminating the variable $t$ in the system formed by \eqref{P} and \eqref{Q}. However, in order to avoid extraneous components (see \cite{Far2} for details), we must divide first $\tilde{P}(x,y,t)$, $\tilde{Q}(x,y,t)$ by their \emph{contents} with respect to $t$.\footnote{Let $f(x_1,\ldots,x_r,x_{r+1},\ldots,x_n)$ be a polynomial in the variables $x_1,\ldots,x_r,x_{r+1},\ldots,x_n$ with coefficients in a unique factorization domain. The \emph{content} $\mbox{cont}_{x_1,\ldots,x_r}(f)$ of $f$ with respect to $x_1,\ldots,x_r$ is the $\gcd$ of the coefficients of $f$, seen as a polynomial in $x_{r+1},\ldots,x_n$ whose coefficients are polynomials in $x_1,\ldots,x_r$. The polynomial $\tilde{f}=\frac{1}{\mbox{cont}_{x_1,\ldots,x_r}(f)}\cdot f$ is called the \emph{primitive part} of $f$ with respect to $x_1,\ldots,x_r$.} Since $\tilde{Q}(x,y,t)$ can be written as 
\[
\begin{array}{rcl}
\tilde{Q}(x,y,t)&=&W^2(t)(x^2+y^2)-2W(t)X(t)x-2W(t)Y(t)y\\
&&+X^2(t)+Y^2(t)-d^2W^2(t),
\end{array}
\]
one can easily see (Lemma 3.1, \cite{Far2}) that the $t$--content of $\tilde{Q}$ is equal to
\begin{equation}\label{cont1}
\mu(t)=\gcd(W(t),X^2(t)+Y^2(t)).
\end{equation}
Similarly (Lemma 3.2, \cite{Far2}) the $t$--content of $\tilde{P}$ is equal to
\begin{equation}\label{cont2}
\beta(t)=\sigma(t) \gamma(t)\mu(t)
\end{equation}
with
\begin{equation}\label{cont3}
\sigma(t)=\gcd(W(t),W'(t)), \mbox{ }\gamma(t)= \gcd(U(t)/\sigma(t), V(t)/\sigma(t) ).
\end{equation}
Let $P(x,y,t)$ and $Q(x,y,t)$ be the polynomials obtained after removing the $t$-contents from $\tilde{P}$ and $\tilde{Q}$,
\begin{equation}\label{PyQ}
P(x,y,t):=\frac{\tilde{P}(x,y,t)}{\beta(t)},\mbox{ }Q(x,y,t):=\frac{\tilde{Q}(x,y,t)}{\mu(t)},
\end{equation}
and let \[H(x,y)=\mbox{Res}_t(P(x,y,t),Q(x,y,t)).\]
If $\mu(t)$ is constant, then $F(x,y)=H(x,y)$. If $\mathrm{deg}(\mu(t))>0$ then $H(x,y)$ can have extraneous, linear factors. Hence $F(x,y)$ is the result of dividing $H(x,y)$ by the product of the extraneous factors. Notice that if $\phi(t)$ is \emph{polynomial}, i.e. if $W(t)=1$, then $\mu(t)$ is constant and no extraneous factors can appear. 

In order to see that $F(x,y)=0$ really corresponds to ${\mathcal O}_d({\mathcal C})$, one observes that $F(x,y)$ is a factor of $\mbox{Res}_t(P(x,y,t),Q(x,y,t))$. By well-known properties of resultants (see \cite{Cox}) then for each $(x_0,y_0)\in {\mathcal O}_d({\mathcal C})$ either there exists $t_0$ such that $P(x_0,y_0,t_0)=Q(x_0,y_0,t_0)=0$, or $(x_0,y_0)$ is a common zero of the leading coefficients $\mbox{lc}_t(P)$, $\mbox{lc}_t(Q)$ of $P(x,y,t)$ and $Q(x,y,t)$ with respect to $t$. The next result, proved in Appendix I so as not to stop the flow of the paper, characterizes the points where both $\mbox{lc}_t(P)$, $\mbox{lc}_t(Q)$ simultaneously vanish. 

\begin{lemma}\label{l-aux}
The only points $(x_0,y_0)$ where the leading coefficients of $P(x,y,t)$, $Q(x,y,t)$ with respect to $t$ simultaneously vanish are $P_{\pm\infty}$, in the case when $P_{\pm\infty}$ are affine points. 
\end{lemma}

\noindent Therefore for every $(x_0,y_0)\in {\mathcal O}_d({\mathcal C})$, with perhaps the exception of $P_{\pm\infty}$, there exists $t_0$ such that $P(x_0,y_0,t_0)=Q(x_0,y_0,t_0)=0$. In turn, this means that $(x_0,y_0)$ simultaneously belongs to the normal line to ${\mathcal C}$ through $p_0= \phi(t_0)$, and to the circle centered at $p_0$ of radius $d$, which implies that $(x_0,y_0)$ is generated by $t_0$ via $\phi_d(t)$. Note that in that case $W(t_0)\neq 0$: indeed, since by hypothesis $\gcd(X(t),Y(t),W(t))=1$, $W(t_0)=0$ implies that either ${\mathcal X}(t)$ or ${\mathcal Y}(t)$ tend to infinity as $t\to t_0$. 

A more algebraic definition of the offset curve, using an incidence diagram, can be found in \cite{Rafa1}. We also summarize some notions and results of \cite{AS97} and \cite{Rafa1}, that we will use in this paper. First, ${\mathcal O}_d({\mathcal C})$ has at most two components. An irreducible component of ${\mathcal O}_d({\mathcal C})$ is said to be \emph{simple} if almost every point of that component is generated by just one point of ${\mathcal C}$; otherwise, the component is called \emph{special}. If ${\mathcal C}$ is properly and rationally parametrized, then ${\mathcal O}_d({\mathcal C})$ is reducible iff $U^2(t)+V^2(t)$, in our notation, is a perfect square (see Corollary 3.4 in \cite{AS97}). Furthermore, if ${\mathcal O}_d({\mathcal C})$ is reducible, then it has two rational components. Finally, if ${\mathcal O}_d({\mathcal C})$ is irreducible then it is simple. 

If ${\mathcal O}_d({\mathcal C})$ is reducible the problem which we address in the paper can be solved in an easier way. Indeed, denoting the rational parametrizations of the components of ${\mathcal O}_d({\mathcal C})$ by $\bfx_1(t)$ and $\bfx_2(t)$, one just needs to: (1) compute the singularities of $\bfx_1(t)$ and $\bfx_2(t)$ separately (see for instance \cite{Rubio}); (2) compute the intersections of the components of ${\mathcal O}_d({\mathcal C})$, for instance by setting $\bfx_1(t)=\bfx_2(s)$ and applying elimination methods. Thus, in the rest of this paper we will assume that ${\mathcal O}_d({\mathcal C})$ is irreducible, in which case it is also simple. Therefore, the real part of ${\mathcal O}_d({\mathcal C})$, which is the object that we want to study here, corresponds to the points generated by $t\in {\Bbb R}$ (see Remark \ref{new} in Subsection \ref{sec-sing}), possibly with the exception of some isolated singularities and $P_{\pm \infty}$.

In this paper we will address the computation of the real affine, non-isolated singularities of ${\mathcal O}_d({\mathcal C})$, which are generated by real values of the parameter of $\phi(t)$ via \eqref{offset}. The only exception to this are the points $P_{\pm \infty}$, which may be generated only by $t=\infty$. Since our goal is to find the $t$ values generating the singularities of ${\mathcal O}_d({\mathcal C})$, this is not really an issue. Nevertheless, if one also wants to detect whether or not $P_{\pm \infty}$ are singular, it suffices to reparametrize the curve so that $P_{\pm \infty}$ are generated by affine values of $t$, and then examine these $t$ values.

\subsection{Subresultants.} \label{sec-subres}

We refer to \cite{Basu}, \cite{Bened}, \cite{GLRR1} and \cite{vonzur} for further reading on the notions and results in this subsection. Let ${\Bbb D}$ be an integral domain, and let $f,g\in {\Bbb D}[t]$ be the polynomials
\[
f(t)=a_nt^n+a_{n-1}t^{n-1}+\cdots+a_0,\mbox{ }g(t)=b_mt^m+b_{m-1}t^{m-1}+\cdots +b_0,
\]
where $\deg(f)\leq n$, $\deg(g)\leq m$.
\begin{definition}\label{defsyl} 
For $i\in\{0,\ldots,\inf(n,m)-1\}$, the Sylvester matrix of index
$i$ associated to $f(t)$, $n$, $g(t)$ and $m$, denoted by $
\mathbf{Sylv}_i(P,n,Q,m)$, is the $(n+m-2i)\times (n+m-i)$ matrix:
\[\hbox{\bf Sylv}_i(f,n,g,m)=\overbrace{  \begin{pmatrix}
                        \begin{matrix}a_n&\ldots&a_0\\
                                   &\ddots&&\ddots& \\
                                   &&a_n&\ldots&a_0\\\end{matrix} \\
                        \begin{matrix}b_m&\ldots&b_0 \\
                                   &\ddots&&\ddots& \\
                                   &&b_m&\ldots&b_0\\\end{matrix} \\
\end{pmatrix}}^{n+m-i}
                   \begin{matrix} \left.\begin{matrix}\\ \\
\\\end{matrix}\right\}&m-i\\
                                  \left.\begin{matrix}\\ \\
\\\end{matrix}\right\}&n-i\\
                   \end{matrix}\]
The Sylvester matrix of index 0 associated to $f(t)$, $n$, $g(t)$
and $m$ is denoted by $\hbox{\bf Sylv}(f,n,g,m)$.
\end{definition}

If $\deg(f)=n$ and $\deg(g)=m$ then the Sylvester matrix of index
$0$ is simply called the \emph{Sylvester matrix of $f(t)$ and $g(t)$},
denoted by $\hbox{\bf Sylv}(f,g)$, and the Sylvester matrix of
index $i\neq 0$ is denoted by $\hbox{\bf Sylv}_i(f,g)$.

\begin{definition}
The determinant of $\hbox{\bf Sylv}(f,g)$, denoted by $\mathrm{Res}(f,g)$, is known as the
\emph{resultant} of $f(t)$ and $g(t)$.
\end{definition}

The concept of determinant polynomial associated to a matrix
provides one of the usual ways to define subresultant polynomials.

\begin{definition}
Let $\Delta$ be a $m\times n$ matrix with $m\leq n$. The
determinant polynomial of $\Delta$, $\hbox{\bf detpol}(\Delta)$,
is defined as:
\[
\hbox{\bf detpol}(\Delta)=\sum_{k=0}^{n-m}\det(\Delta_k)t^{n-m-k}
\]
where $\Delta_k$ is the square submatrix of $\Delta$ consisting of
the first $m-1$ columns and the $(k+m)$--th column.
\end{definition}

\begin{definition}
The $i$-th subresultant polynomial of $f(t)$ and $g(t)$, $\mathbf{Subres}_i(f,n,g,m)$, is the determinant polynomial of $\mathbf{Sylv}_i(f,n,g,m)$, i.e.
$$
\mathbf{Subres}_i(f,n,g,m)=\mathbf{detpol}(\mathbf{Sylv}_i(f,n,g,m)), \text{ for } 0\leq i \leq \inf(n,m)-1.
$$
The sequence $\{\mathbf{Subres}_i(f,n,g,m)\}_{i\geq 0}$ is called the subresultant chain of $f,g$.
\end{definition}

We have that $\deg(\mathbf{Subres}_i(f,n,g,m) )\leq i$. The coefficient of degree $i$ of the polynomial $\mathbf{Subres}_i(f,n,g,m)$, denoted by $\mathbf{sres}_i(f,n,g,m)$, is called the \emph{$i$-th principal subresultant coefficient} of $f,n$ and $g,m$. When $\mathbf{sres}_i(f,n,g,m)=0$, the polynomial $\mathbf{Subres}_i(f,n,g,m)$ is said to be \emph{defective}.  

\begin{theorem}\label{th-1}[Fundamental Property of subresultants] Let $\mathbb{K}$ be the fraction field of $\mathbb{D}$, and assume that $\deg(f)=n$ or $\deg(g)=m$, $\gcd(f,g)\neq f$ and $\gcd(f,g)\neq g$. Then the first subresultant polynomial different from zero in the sequence $\{\mathbf{Subres}_k(f,n,g,m)\}_{k\geq 0}$ is non-defective, and equal to the greatest common divisor of $f,g$ in $\mathbb{K}[t]$. 
\end{theorem}

\begin{theorem}\label{th-2} Let $\psi  : A \rightarrow  B $ be a ring homomorphism and $P, Q \in A[t] $ be two polynomials with $\deg(P) = p$ and $\deg(Q) = q$. If $\deg(\psi(P)) = p$ and $\deg(\psi(Q)) = q$ then for any $i \leq q$ we have $\mathbf{Subres}_i(\psi(P),p,\psi(Q),q)  =\psi(\mathbf{Subres}_i(P, p,Q,q)).$ If  $\deg(\psi(P)) = p$ and $\deg(\psi(Q))=q^* < q$ (or vice versa), $\mathbf{Subres}_i(\psi(P),p,\psi(Q),q^*)$ and $\psi(\mathbf{Subres}_i(P, p,Q,q))$ are proportional.
\end{theorem}

  
  
  



\subsection{Singularities.} \label{sec-sing}

Let ${\mathcal D}$ be an algebraic planar curve, implicitly defined by $f(x,y)=0$. An affine point $S\in {\mathcal D}$ is a \emph{singularity} iff $f_x(S)=f_y(S)=0$; a nonsingular point of ${\mathcal D}$ is said to be a \emph{regular} point. One can always compute a local parametrization $(x(h),y(h))$ of ${\mathcal D}$ around any point $S\in {\mathcal D}$, regular or singular, where $x(h),y(h)$ are two analytic functions in a neighborhood of $h=0$. An equivalence class of irreducible local parametrizations around $S$ is called a \emph{place} \cite{walker}; we say that the place is \emph{centered} at $S=(x(0),y(0))$. Furthermore, we say that the place is \emph{real} if there is a representative of the class where all the coefficients are real. If $(x(h),y(h))$, with $x(h),y(h)$ real, represents a real place, where $x(h),y(h)$ converge for $|h|<\epsilon$, the set of points of ${\mathcal D}$ defined by $(x(h),y(h))$ for $|h|<\epsilon$, is called a \emph{real branch} of ${\mathcal D}$ through $(x(0),y(0))$. 

If $S$ is a self-intersection of ${\mathcal D}$, then $S$ is the center of several, different, places of ${\mathcal D}$. If ${\mathcal D}$ is not parallel to the $y$-axis, one can prove \cite{AS07} that any place centered at $S$ can be written, in a coordinate system centered at $S$ where the $x$-axis coincides with the tangent to ${\mathcal D}$ at $S$ (see Figure \ref{fig:new}, left), as 
\[{\mathcal P}(h)=(h^p,\beta_qh^q+\cdots),\]with $p,q\in {\Bbb N}$, $p\geq 1$, $q>p$. We will say that ${\mathcal P}(h)$ is \emph{singular} if $p\geq 2$. A point $S\in {\mathcal D}$ is a singularity iff it is either the center of one singular place, or the center of several singular or regular places. In the first case, we will say that $S$ is a \emph{local singularity} of ${\mathcal D}$. In the second case, we will say that $S$ is a \emph{self-intersection} of ${\mathcal D}$. Notice that $S$ can simultaneously be a local singularity and a self-intersection when it is the center of several places, and at least one of them is singular (see Figure \ref{fig:new}, bottom-right). 

\begin{remark}\label{new}
In the rest of the paper we will address the computation of the real local singularities and the real self-intersections that are not isolated, i.e. that are the center of at least one real place. Since we are assuming that $\mathcal{O}_d({\mathcal C})$ is simple, then any real branch of $\mathcal{O}_d({\mathcal C})$ comes from a real branch of ${\mathcal C}$ \cite{AS07}. Since in turn every real branch of ${\mathcal C}$ is generated by real values of $t$ via $\phi(t)$, any affine $t$ value generating a real, non-isolated singularity of $\mathcal{O}_d({\mathcal C})$ is real.
\end{remark} 

\begin{figure}
\begin{center}
\includegraphics[scale=0.4]{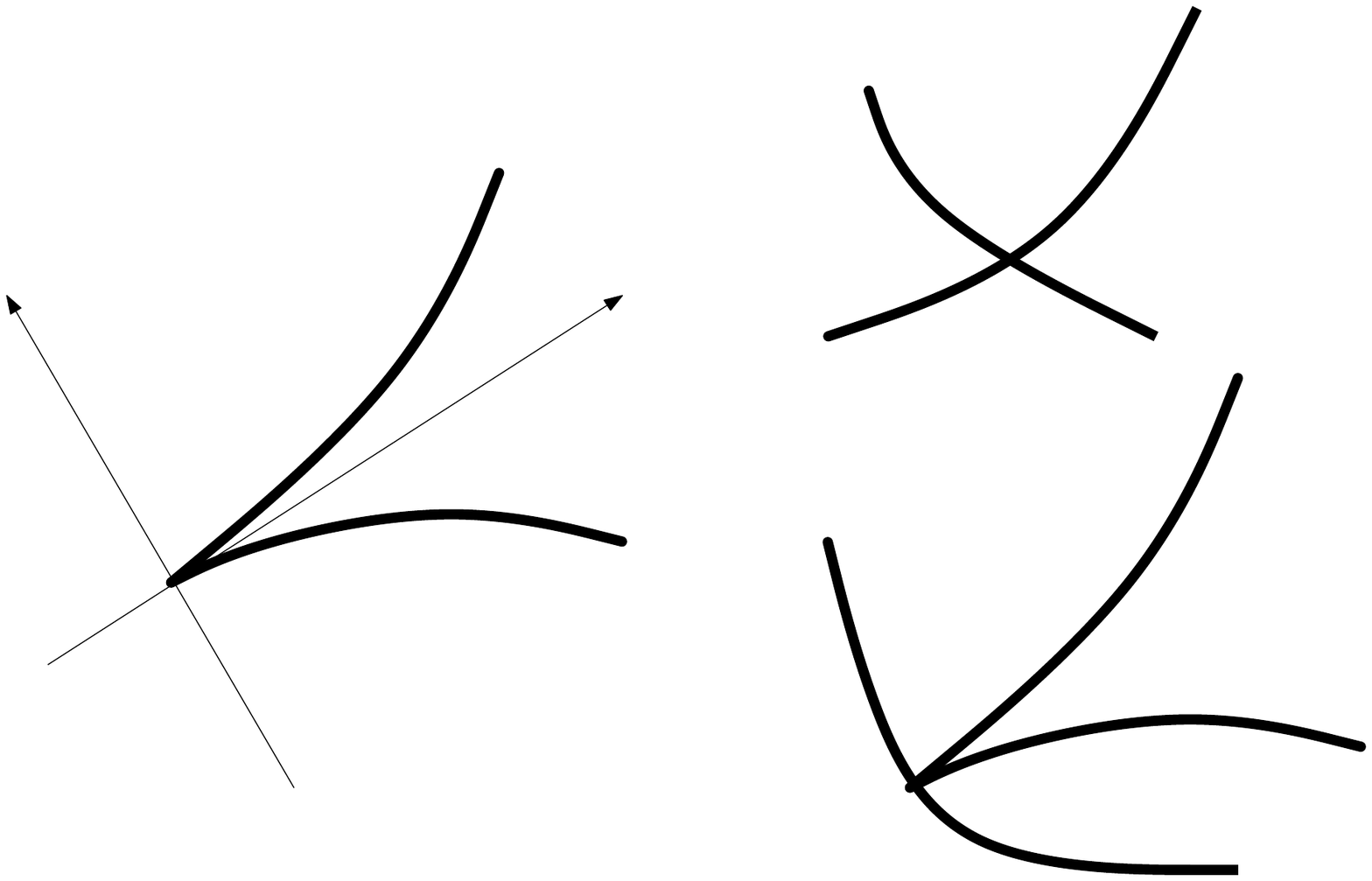}
\end{center}
\caption{Local singularities and self-intersections.}\label{fig:new}
\end{figure}
 
\section{Strategy for computing the offset singularities.} \label{sec-inv}

\subsection{The idea.}\label{sec-main}

Let us consider first the computation of the real self-intersections of ${\mathcal O}_d({\mathcal C})$. Later on we will see that our method to compute these singularities provides also, as a by-product, the local, real singularities of ${\mathcal O}_d({\mathcal C})$. Our idea to find the self-intersections of ${\mathcal O}_d({\mathcal C})$ is inspired in the strategy used in \cite{Fukushima} to find the genus of ${\mathcal O}_d({\mathcal C})$. In \cite{Fukushima}, the curve implicitly defined in the $(t,\alpha)$ plane by
\[\alpha^2-(U^2(t)+V^2(t))=0,\]is introduced. In order to avoid difficulties in the analysis of the $t$-values generating local singularities of ${\mathcal C}$, we will use instead the curve ${\mathcal M}$ defined as
\[
\alpha^2-(\widehat{U}^2(t)+\widehat{V}^2(t))=0,
\]
with 
$$\widehat{U} = \frac{U}{\gcd(U,V)} \text{ and }\widehat{V} = \frac{V}{\gcd(U,V)}.$$

\noindent Notice here that $t$ is the parameter in the parametrization $\phi(t)$, and $\alpha$ is a new auxiliary variable; a 
similar construction was used in \cite{AS97}, see Definition 3.2 therein, to analyze the rationality of the offset. 

We will restrict to the case when ${\mathcal M}$ is irreducible. Notice that ${\mathcal M}$ is reducible iff $\widehat{U}^2(t)+\widehat{V}^2(t)$ is a perfect square, i.e. iff $U^2(t)+V^2(t)$ is a perfect square\footnote{In order to check whether or not $U^2(t)+V^2(t)$ is a perfect square, we observe that non-negative elements of the ground field, i.e. non-negative real numbers, can be regarded as perfect squares.}, in which case ${\mathcal O}_d({\mathcal C})$ has two rational components. However this case admits an easier solution, as we observed in Section \ref{subsec-off}. Furthermore, as we also observed in Section \ref{subsec-off}, if ${\mathcal M}$ is irreducible then ${\mathcal O}_d({\mathcal C})$ is irreducible, and therefore ${\mathcal O}_d({\mathcal C})$ is also simple. 

Now ${\mathcal O}_d({\mathcal C})$ can be seen as the image of ${\mathcal M}$ under the following rational transformation of the plane, denoted by $\varphi(t,\alpha)$ (see Figure \ref{fig:inverse}):
\begin{equation}\label{t1}
x=x(t,\alpha)=\frac{X(t)}{W(t)}+d\frac{\widehat{V}(t)}{\alpha},\mbox{ }y=y(t,\alpha)=\frac{Y(t)}{W(t)}-d\frac{\widehat{U}(t)}{\alpha}
\end{equation}
Since under our hypotheses ${\mathcal O}_d({\mathcal C})$ is simple, for almost all points $(x_0,y_0)\in {\mathcal O}_d({\mathcal C})$ there exists a unique $(t_0,\alpha_0)$ such that $(x_0,y_0)=\varphi(t_0,\alpha_0)$. Therefore, the inverse $\varphi^{-1}(x,y)$ exists for almost all points of ${\mathcal O}_d({\mathcal C})$, and is rational as well. Hence if $\mathcal{O}_d(\mathcal{C})$ is simple then $\varphi$ defines a birational transformation between ${\mathcal M}$ and $\mathcal{O}_d(\mathcal{C})$.

\begin{figure}
\begin{center}
\includegraphics[scale=0.5]{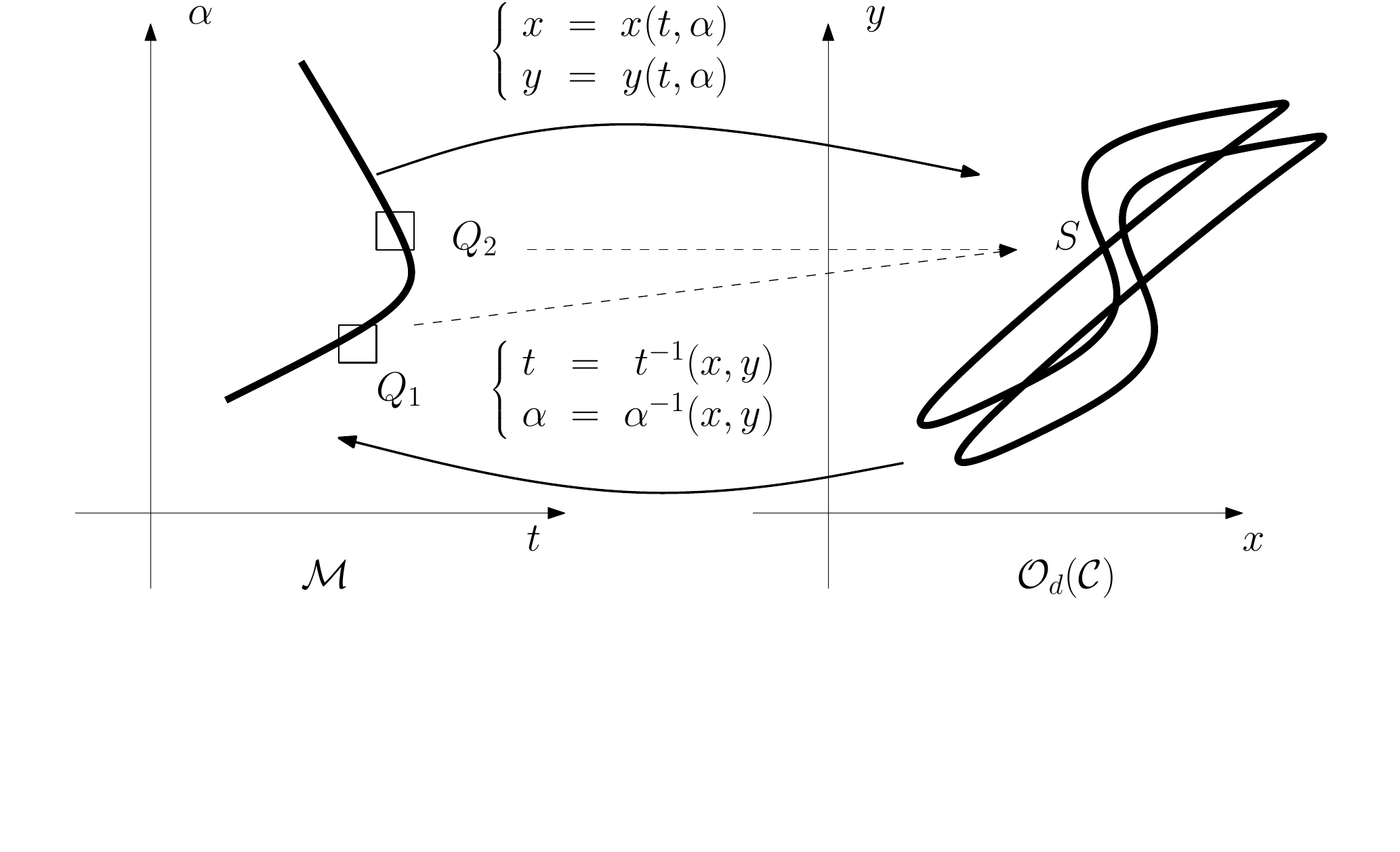}
\end{center}
\caption{Main idea to find the self-intersections of the offset.}\label{fig:inverse}
\end{figure}

Let us write
\begin{equation} \label{inv}
\varphi^{-1}(x,y)=\left(t^{-1}(x,y),\alpha^{-1}(x,y)\right),
\end{equation}
where $t^{-1}(x,y)$ and $\alpha^{-1}(x,y)$ are rational functions. Furthermore, let \[t^{-1}(x,y)=\frac{A(x,y)}{B(x,y)}.\]One can observe that if $S\in  {\mathcal O}_d({\mathcal C})$ is a self-intersection not generated by $t=\infty$, then there are at least two different values $t_1,t_2$ of $t$, and therefore two different points $Q_1,Q_2\in {\mathcal M}$, such that $\varphi(Q_1)=\varphi(Q_2)$ (see Figure \ref{fig:inverse}). So $\varphi^{-1}(S)$ cannot be defined, and therefore the denominator $ B(x,y)$ of $t^{-1}(x,y)$ must vanish at $S$. Hence, the self-intersections of ${\mathcal O}_d({\mathcal C})$ are among the points of ${\mathcal O}_d({\mathcal C})$ where $B(x,y)=0$. Since ${\mathcal O}_d({\mathcal C})$ is ``parametrized" by $\varphi(t,\alpha)$, one does not need to know or make use of the implicit equation of ${\mathcal O}_d({\mathcal C})$ in order to find the points of ${\mathcal O}_d({\mathcal C})$ where $B(x,y)=0$. Instead, one imposes
\begin{equation} \label{check}
 B(x(t,\alpha),y(t,\alpha))=0.
\end{equation}
By repeatedly using that $\alpha^2$ is a polynomial in the variable $t$, the numerator of \eqref{check} leads to a polynomial equation in $t$. The set of real roots of the polynomial provides a list of real $t$-values generating the real self-intersections of ${\mathcal O}_d({\mathcal C})$; notice that since ${\mathcal O}_d({\mathcal C})$ is reduced, this list is finite. We will see later, in Section \ref{subsec-comp}, that in fact this list contains all the real non-isolated singularities of ${\mathcal O}_d({\mathcal C})$, not only the real self-intersections.

In order to compute  $t^{-1}(x,y)$, let us observe the following. For a particular $(x_0,y_0)\in {\mathcal O}_d({\mathcal C})$, $t_0=t^{-1}(x_0,y_0)$ should be the unique root of \[\gcd(P(x_0,y_0,t),\,Q(x_0,y_0,t)),\]for the polynomials $P,Q$ introduced in Equation \eqref{PyQ}. Thus, in order to determine the function $t^{-1}(x,y)$, one can compute the $\gcd$ of $P(x,y,t)$, $Q(x,y,t)$ for a generic point $(x,y)$, considering $P(x,y,t)$, $Q(x,y,t)$ as elements of ${\Bbb R}[x,y][t]$ (i.e. as polynomials in the variable $t$ whose coefficients are real polynomials in $x,y$), with the additional condition $F(x,y)=0$; recall that $F(x,y)$ represents the implicit equation of the offset. More formally, one sees $P,\, Q$ as elements of ${\Bbb R}({\mathcal O}_d({\mathcal C}))[t]$, where ${\Bbb R}({\mathcal O}_d({\mathcal C}))$ is the field of real rational functions of ${\mathcal O}_d({\mathcal C})$ (see \cite{walker}). Since ${\mathcal O}_d({\mathcal C})$ is assumed to be irreducible then ${\Bbb R}({\mathcal O}_d({\mathcal C}))[t]$ is a Euclidean domain. Therefore 
\[{\mathcal G}(x,y,t)=\gcd_{{\Bbb R}({\mathcal O}_d({\mathcal C}))[t]}(P,Q)\]is well defined and can be computed by means of the Euclidean algorithm. In \cite{Fukushima}, ${\mathcal G}(x,y,t)$ is introduced in a similar way, and is computed using the Euclidean algorithm.

\subsection{Computation of ${\mathcal G}(x,y,t)$ and $t^{-1}(x,y)$ via subresultants.}\label{subsec-comp}

In order to compute ${\mathcal G}(x,y,t)$ by means of the Euclidean algorithm, we must perform several divisions between elements of ${\Bbb R}({\mathcal O}_d({\mathcal C}))[t]$. Furthermore, we need to check, at each step, if the remainder is zero in ${\Bbb R}({\mathcal O}_d({\mathcal C}))[t]$. This implies either computing $F(x,y)$ first and then checking if $F(x,y)$ divides each remainder, which can be extremely costly, or substituting $x=x(t,\alpha)$, $y=y(t,\alpha)$ in each remainder and then checking if the result is divisible by $\alpha^2-(\widehat{U}^2(t)+\widehat{V}^2(t))$. 

In this section we will present an alternative, faster method based on subresultants. For this purpose we need several previous results. On the one hand, these results allow us to prove that ${\mathcal G}(x,y,t)$, as a polynomial in the variable $t$ with coefficients in $x,y$, has just one root; furthermore, we will see that this root is simple, so that the degree of ${\mathcal G}(x,y,t)$ in the variable $t$ is 1. On the other hand, these results will be used later to show that our method provides all the real singularities of ${\mathcal O}_d({\mathcal C})$, not only the self-intersections. For the next lemma we recall the notation $\mathcal{X}(t) = \frac{X(t)}{W(t)}$, $\mathcal{Y}(t) = \frac{Y(t)}{W(t)}$ introduced in Subsection \ref{sec-prelim}.

\begin{lemma} \label{lem-1}
Let $(x_0,y_0)\in {\mathcal O}_d({\mathcal C})$, $(x_0,y_0)\neq P_{\pm \infty}$, which is not a self-intersection of ${\mathcal O}_d({\mathcal C})$, be an affine point 
generated by a real value $t_0$ with $(U(t_0),V(t_0))\neq (0,0)$. Then $\frac{\partial P}{\partial t}(x_0,y_0,t_0)=0$ if and only if $(x_0,y_0)$ is a local singularity of ${\mathcal O}_d({\mathcal C})$.
\end{lemma}

\begin{proof} 
Let us assume that $(x_0,y_0)$ belongs to the exterior offset; we can argue in a similar way if $(x_0,y_0)$ belongs to the interior offset. Since $(x_0,y_0)\in {\mathcal O}_d({\mathcal C})$ and $(x_0,y_0)$ is generated by $t_0$, then 
\[P(x_0,y_0,t_0)=Q(x_0,y_0,t_0)=0\]and therefore
\[\tilde{P}(x_0,y_0,t_0)=\tilde{Q}(x_0,y_0,t_0)=0.\]
Furthermore, let
\begin{equation} \label{laP}
P^{\star}(x,y,t)=(x-\mathcal{X}(t))\cdot \mathcal{X}'(t)+(y-\mathcal{Y}(t))\cdot \mathcal{Y}'(t).
\end{equation}
A direct calculation shows that 
\begin{equation} \label{Pstar}
P^{\star}(x,y,t)=\frac{1}{W^3(t)}\cdot \tilde{P}(x,y,t).
\end{equation}
Furthermore, since $(x_0,y_0)$ is affine $W(t_0)\neq 0$. Since additionally $\tilde{P}(x_0,y_0,t_0)=0$, we get $P^{\star}(x_0,y_0,t_0)=0$. Furthermore, by differentiating in \eqref{Pstar} we have 
\[\frac{\partial P^{\star}}{\partial t}=\frac{-3W'(t)}{W^4(t)}\cdot \tilde{P}(x,y,t)+\frac{1}{W^3(t)}\cdot \frac{\partial \tilde{P}}{\partial t}.\]
Hence, since $\tilde{P}(x_0,y_0,t_0)=0$ and $W(t_0)\neq 0$, $\frac{\partial P^{\star}}{\partial t}(x_0,y_0,t_0)=0$ iff $\frac{\partial \tilde{P}}{\partial t}(x_0,y_0,t_0)=0$. So let us see that, under the considered hypotheses, $\frac{\partial P^{\star}}{\partial t}(x_0,y_0,t_0)=0$ iff $(x_0,y_0)$ is a local singularity of ${\mathcal O}_d({\mathcal C})$. 

By differentiating \eqref{laP} with respect to $t$, and evaluating at $(x_0,y_0,t_0)$, we get 
\[
\frac{\partial P^{\star}}{\partial t}(x_0,y_0,t_0)=(x_0-\mathcal{X}(t_0))\cdot \mathcal{X}''(t_0)+(y_0-\mathcal{Y}(t_0))\cdot \mathcal{Y}''(t_0)-\left((\mathcal{X}'(t_0))^2+(\mathcal{Y}'(t_0))^2\right).
\]
Thus $\frac{\partial P^{\star}}{\partial t}(x_0,y_0,t_0)=0$ is equivalent to
\begin{equation} \label{eszero}
(x_0-\mathcal{X}(t_0))\cdot \mathcal{X}''(t_0)+(y_0-\mathcal{Y}(t_0))\cdot \mathcal{Y}''(t_0)=(\mathcal{X}'(t_0))^2+(\mathcal{Y}'(t_0))^2.
\end{equation}
Since $(U(t_0),V(t_0))\neq (0,0)$ by hypothesis, $(\mathcal{X}'(t_0),\mathcal{Y}'(t_0))\neq 0$. Therefore the normal vector to ${\mathcal C}$ at the point $({\mathcal X}(t_0),{\mathcal Y}(t_0))$ is $\pm(\mathcal{Y}'(t_0),-\mathcal{X}'(t_0))$, where we consider the $+$ sign if $(x_0,y_0)$ belongs to the exterior offset, as it is our case, and the $-$ sign if $(x_0,y_0)$ belongs to the interior offset. Additionally, since $\tilde{Q}(x_0,y_0,t_0)=0$, from the definition of $\tilde{Q}(x,y,t)$ we get that the modulus of $(x_0-\mathcal{X}(t_0),y_0-\mathcal{Y}(t_0))$ is equal to $d$. Hence under our hypotheses,
\begin{equation} \label{vector}
(x_0-\mathcal{X}(t_0),y_0-\mathcal{Y}(t_0))=\frac{d}{\sqrt{\mathcal{X}'^2(t_0)+\mathcal{Y}'^2(t_0)}}\cdot (\mathcal{Y}'(t_0),-\mathcal{X}'(t_0)).
\end{equation}
The right hand-side of \eqref{eszero} is equal to the dot product of $(x_0-\mathcal{X}(t_0),y_0-\mathcal{Y}(t_0))$ and $(\mathcal{X}''(t_0),\mathcal{Y}''(t_0))$. Therefore, taking \eqref{vector} into account, we get that $\frac{\partial P^{\star}}{\partial t}(x_0,y_0,t_0)=0$ is equivalent to
\begin{equation} \label{cond}
\frac{\mathcal{X}'(t_0)\mathcal{Y}''(t_0)-\mathcal{X}''(t_0)\mathcal{Y}'(t_0)}{\left[\mathcal{X}'^2(t_0)+\mathcal{Y}'^2(t_0)\right]^{3/2}}=-\frac{1}{d}.
\end{equation}
This equality can be written as $k(t_0)=-\frac{1}{d}$, where $k(t_0)$ is the curvature of ${\mathcal C}$ at the point $(\mathcal{X}(t_0),\mathcal{Y}(t_0))$. But under our hypotheses, $k(t_0)=-\frac{1}{d}$ is equivalent to $(x_0,y_0)$ being a local singularity of ${\mathcal O}_d({\mathcal C})$ (see page 163 of \cite{Farouki}).

\end{proof}

Let $(x_0,y_0)\in {\mathcal O}_d({\mathcal C})$, and let $G_{x_0,y_0}(t)=\mbox{gcd}(P(x_0,y_0,t),Q(x_0,y_0,t))$. 
The roots of $G_{x_0,y_0}(t)$ are exactly the affine $t$-values generating $(x_0,y_0)$ via $\phi_d(t)$. The behavior of $G_{x_0,y_0}(t)$ is analyzed in the following result, which follows from Lemma \ref{lem-1}.

\begin{proposition} \label{cor1}
Let $(x_0,y_0)\in {\mathcal O}_d({\mathcal C})$, $(x_0,y_0)\neq P_{\pm\infty}$, be generated by some real $t$-value $t_0$ satisfying $(U(t_0),V(t_0))\neq (0,0)$. If $(x_0,y_0)$ is a singularity of ${\mathcal O}_d({\mathcal C})$ then $\deg(G_{x_0,y_0}(t))>1$.
\end{proposition}

\begin{proof} If $(x_0,y_0)$ is a self-intersection of ${\mathcal O}_d({\mathcal C})$, since by hypothesis $(x_0,y_0)\neq P_{\pm\infty}$, it must be generated by at least two $t$-values; therefore, $\deg(G_{x_0,y_0}(t))>1$. Suppose now that $(x_0,y_0)$ is a local singularity of ${\mathcal O}_d({\mathcal C})$, not a self-intersection. Since $(x_0,y_0)$ is generated by $t_0$, it follows that $t_0$ is a common root of $P(x_0,y_0,t)$ and $Q(x_0,y_0,t)$. Moreover, by Lemma \ref{lem-1} 
\begin{equation}\label{esta}
P(x_0,y_0,t_0)= \frac{\partial P}{\partial t}(x_0,y_0,t_0)=0.
\end{equation} 
Following the same argument as in the proof of Lemma \ref{lem-1}, we get that \eqref{esta} is equivalent to $$P^{\star}(x_0,y_0,t_0) = \frac{\partial P^{\star}}{\partial t}(x_0,y_0,t_0)=0.$$ Additionally, $Q(x,y,t)$ is the primitive part of the numerator of
 \begin{equation}\label{qstart }
Q^{\star}(x,y,t)=(x-\mathcal{X}(t))^2+(y-\mathcal{Y}(t))^2-d^2, 
\end{equation}
and $\frac{\partial Q^{\star}}{\partial t}=2P^{\star}$; in fact, one can check that the numerator of $Q^{\star}(x,y,t)$ is $\tilde{Q}(x,y,t)$. Hence,
\begin{equation}\label{qstar}
Q^{\star}(x_0,y_0,t_0)=\frac{\partial Q^{\star}}{\partial t}(x_0,y_0,t_0)=0.
\end{equation}
Furthermore, since $W(t_0)\neq 0$ the content of $\tilde{Q}(x,y,t)$ cannot vanish at $t=t_0$, and therefore \eqref{qstar} implies 
$Q(x_0,y_0,t_0)=\frac{\partial Q}{\partial t}(x_0,y_0,t_0)=0$. As a consequence, $t_0$ is also a root of both $\frac{\partial P }{\partial t}(x_0,y_0,t)$ and $\frac{\partial Q }{\partial t}(x_0,y_0,t)$. Thus the multiplicity of $t_0$ as a root of $G_{x_0,y_0}(t)$ is greater than 1.
\end{proof}


\begin{remark} \label{conv-false}
If $(x_0,y_0)\neq P_{\pm \infty}$ is a real, non-isolated point of ${\mathcal O}_d({\mathcal C})$, whenever $\deg(G_{x_0,y_0}(t))>1$ then $(x_0,y_0)$ must be a singularity of the curve ${\mathcal H}$ defined by $H(x,y)=\mathrm{Res}_t(P(x,y,t),Q(x,y,t))$. Indeed, if $G_{x_0,y_0}(t)$ has just one multiple root $t_0$ then since $G_{x_0,y_0}(t)$ is a real polynomial, $t_0\in {\Bbb R}$; but then $(x_0,y_0)$ must be a local singularity because of Lemma \ref{lem-1}. If 
$G_{x_0,y_0}(t)$ has different roots then $(x_0,y_0)$ is a self-intersection of ${\mathcal H}$. By Proposition \ref{libre} in Appendix III, the component of ${\mathcal H}$ defining ${\mathcal O}_d({\mathcal C})$ is square-free, i.e. $F(x,y)$  (see page 7) is the polynomial of \emph{minimum} degree defining ${\mathcal O}_d({\mathcal C})$. Therefore if $G_{x_0,y_0}(t)$, with $(x_0,y_0)\in{\mathcal O}_d({\mathcal C})$, has different roots then $(x_0,y_0)$ is either a self-intersection of ${\mathcal O}_d({\mathcal C})$, or the intersection of ${\mathcal O}_d({\mathcal C})$ with some spurious factor (see Section \ref{subsec-off}). So if $\deg(G_{x_0,y_0}(t))>1$ the point $(x_0,y_0)$ is either a singularity of ${\mathcal O}_d({\mathcal C})$, or an intersection point between ${\mathcal O}_d({\mathcal C})$ and a spurious factor of ${\mathcal H}$. Hence, the converse of Proposition \ref{cor1} is not true, in general. 
\end{remark}

From Remark \ref{conv-false}, we get that $\deg(G_{x_0,y_0}(t))=1$ for almost all points of $(x_0,y_0)\in {\mathcal O}_d({\mathcal C})$. Hence, the following result follows.

\begin{theorem}\label{deg-1}
The degree of ${\mathcal G}(x,y,t)$ in the $t$ variable, is equal to 1.
\end{theorem} 


From Theorem \ref{th-1} and Theorem \ref{deg-1}, one has that 
\[{\mathcal G}(x,y,t)=\mathbf{Subres}_1({\mathcal P},n,{\mathcal Q},m)(t),\]where ${\mathcal P},{\mathcal Q}$ represent the polynomials $P,Q$ seen as elements of ${{\Bbb R}({\mathcal O}_d({\mathcal C}))}$, and 
$n,m$ are the degrees of ${\mathcal P},{\mathcal Q}$ in ${{\Bbb R}({\mathcal O}_d({\mathcal C}))}$. However, from Lemma \ref{l-aux} we have $n=\deg_t(P(x,y,t))$ and $m=\deg_t(Q(x,y,t))$; in other words, the leading coefficients of the polynomials $P(x,y,t)$ and $Q(x,y,t)$ with respect to $t$ are not multiples of $F(x,y)$, the implicit equation of ${\mathcal O}_d({\mathcal C})$. Therefore, by Theorem \ref{th-2}, in order to compute ${\mathcal G}(x,y,t)$ we can compute $\mathbf{Subres}_1(P,n,Q,m)$ in the domain ${\Bbb R}[x,y][t]$, and then consider the coefficients of the resulting polynomial modulo $F(x,y)$. After writing  
\[
\mathbf{Subres}_1(P,n,Q,m)(t)=\mathbf{sres}_1(x,y)\,t+\mbox{sr}(x,y),
\]
we arrive at the following result.

\begin{theorem} \label{main} The inverse mapping \eqref{inv} satisfies
\begin{equation} \label{t-inv}
t^{-1}(x,y)=-\frac{\mathrm{sr}(x,y)}{\mathbf{sres}_1(x,y)}
\end{equation}
for $(x,y)\in {\mathcal O}_d({\mathcal C})$.
 
\end{theorem}

\begin{proof} Let ${\mathcal G}(x,y,t)=B(x,y)t+A(x,y)$. Then $t^{-1}(x,y)=-\frac{A(x,y)}{B(x,y)}$. However since ${\mathcal G}(x,y,t)=\mathbf{Subres}_1(P,n,Q,m)(t)\mbox{ mod }F(x,y)$, we have 
\[B(x,y)=\mathbf{sres}_1(x,y) \mbox{ mod }F(x,y)\]and \[A(x,y)=\mbox{sr}(x,y)\mbox{ mod }F(x,y).\]But then for $(x,y)\in {\mathcal O}_d({\mathcal C})$, i.e. whenever $F(x,y)=0$, we have \eqref{t-inv}.
\end{proof}

\begin{remark} The function $t^{-1}(x,y)$ is not necessarily defined at every point of ${\mathcal O}_d({\mathcal C})$. In fact, this is the crucial idea to find the self-intersections of ${\mathcal O}_d({\mathcal C})$. 
\end{remark}

\section{Main result and complete algorithm.}\label{final-th}

The following result, derived from the ideas and results in Section \ref{sec-inv}, shows how to compute the real singularities of $\mathcal{O}_d(\mathcal{C})$ not coming from local singularities of ${\mathcal C}$. 

 \begin{proposition}\label{mainpropo}
If $(x_0,y_0)\in {\mathcal O}_d({\mathcal C})$ is a non-isolated, real affine singularity of $\mathcal{O}_d(\mathcal{C})$, generated by $t_0\in {\Bbb R}$, with $(U(t_0),V(t_0))\neq (0,0)$, then there exists $\alpha_0\in {\Bbb R}$, $\alpha^2_0=\widehat{U}^2(t_0)+\widehat{V}^2(t_0)$, such that $(x_0,y_0)=(x(t_0,\alpha_0),y(t_0,\alpha_0))$ and
\begin{equation}\label{eq-fundam}
\mathbf{sres}_1(x(t_0,\alpha_0),y(t_0,\alpha_0))=0.
\end{equation}
\end{proposition}

\begin{proof}
Let $(x_0,y_0)$ be a singularity of $\mathcal{O}_d(\mathcal{C})$. Since $(x_0,y_0)$  is generated by a real value $t_0$, there exists $\alpha_0$ such that $(t_0,\alpha_0)$ in $\mathcal{M}$ and $(x_0,y_0) = (x(t_0,\alpha_0),y(t_0,\alpha_0))$. 
If $(x_0,y_0)\neq P_{\pm\infty}$, then we have $\mbox{deg}(G_{x_0,y_0}(t))>1$ by Proposition \ref{cor1}. On the other hand, by Lemma \ref{l-aux} we deduce that either $\mbox{deg}(P(x_0,y_0,t))=\mbox{deg}_t(P(x,y,t))$ or $\mbox{deg}(Q(x_0,y_0,t))=\mbox{deg}_t(Q(x,y,t))$. Furthermore, from Theorem \ref{th-2} and  Lemma \ref{l-aux}, 
the subresultant chain of $P,Q$ specializes well for $x=x_0,y=y_0$ (up to, perhaps, a constant). From Theorem \ref{th-1}, we can compute $G_{x_0,y_0}(t)$ as the first nonzero subresultant in this sequence. Since $\mbox{deg}(G_{x_0,y_0}(t))>1$, we deduce that $\mathbf{sres}_1(x_0,y_0)$ must vanish. If $(x_0,y_0)=P_{\pm \infty}$, which means that $P_{\pm \infty}$ is also generated by at least one real $t$,  then $\mathbf{sres}_1(x_0,y_0)=0$ by Lemma \ref{l-aux} and Definition \ref{defsyl}.
\end{proof} 

Therefore, if ${\mathcal C}$ does not have any local singularities, Proposition \ref{mainpropo} provides a method to find all the singularities of the offset. However if ${\mathcal C}$ has local singularities, then it remains to check that the singularities of ${\mathcal O}_d({\mathcal C})$ coming from local singularities of $\mathcal C$ also satisfy Equation \eqref{eq-fundam}. A proof of this fact is given in Appendix II. As a consequence, the following theorem, which is the main result of the paper, holds.

\begin{theorem}\label{bigth}
If $(x_0,y_0)\in {\mathcal O}_d({\mathcal C})$ is a non-isolated, real affine singularity of $\mathcal{O}_d(\mathcal{C})$, generated by $t_0\in {\Bbb R}$, then there exists $\alpha_0\in {\Bbb R}$, $\alpha^2_0=\widehat{U}^2(t_0)+\widehat{V}^2(t_0)$, such that $(x_0,y_0)=(x(t_0,\alpha_0),y(t_0,\alpha_0))$ and $\mathbf{sres}_1(x(t_0,\alpha_0),y(t_0,\alpha_0))=0$.
\end{theorem}

\noindent From Theorem \ref{bigth}, the real $t$ values giving rise to the real, non-isolated singularities of ${\mathcal O}_d({\mathcal C})$ are among the solutions of the system:
\[
\mathbf{sres}_1(x(t,\alpha),y(t,\alpha))=0, \mbox{ } \alpha^2=\widehat{U}^2(t)+\widehat{V}^2(t).
\]
Substituting $\alpha^2=\widehat{U}^2(t)+\widehat{V}^2(t)$ in $\mathbf{sres}_1(x(t,\alpha),y(t,\alpha))=0$ leads to a polynomial equation of the type
\begin{equation} \label{uno}
\xi_1(t)\,\alpha+\eta_1(t)=0.  
\end{equation}
Now squaring and using again that $\alpha^2=\widehat{U}^2(t)+\widehat{V}^2(t)$, we arrive at a univariate polynomial equation
\begin{equation}\label{pol}
\tilde\omega(t):=\xi_1^2(t)\,(\widehat{U}^2(t)+\widehat{V}^2(t))-\eta_1^2(t)=0.
\end{equation}
Note that the polynomial $\tilde{\omega}(t)$ cannot be identically zero because of Theorem \ref{deg-1}. Let $\omega^{\star}(t)$ be the square-free part of $\tilde{\omega}(t)$, and let $\omega(t)=\frac{\omega^{\star}(t)}{\gcd(\omega^{\star}(t),W(t))}$. Since $\omega(t)$ is not identically zero, the set ${\mathcal B}$ of real roots of $\omega(t)$ is finite. Hence, we deduce the following algorithm {\tt OffsetSing} to compute a finite set containing the $t$-values generating the real, non-isolated singularities (self-intersections and local singularities) of ${\mathcal O}_d({\mathcal C})$. The algorithm requests ${\mathcal O}_d({\mathcal C})$ to be irreducible, which happens iff $U^2(t)+V^2(t)$ is a perfect square.

\begin{algorithm}{{\bf Algorithm} {\tt OffsetSing}}
\begin{algorithmic}[1]
\REQUIRE A proper parametrization $\phi(t)$ of a planar curve ${\mathcal C}$, and an offsetting distance $d>0$ such that ${\mathcal O}_d({\mathcal C})$ is irreducible.
\ENSURE A finite set ${\mathcal B}$ containing the $t$-values generating the real, non-isolated singularities of ${\mathcal O}_d({\mathcal C})$.
\STATE Let $P(x,y,t)$, $Q(x,y,t)$ be the primitive parts of \eqref{P}, \eqref{Q}, respectively.
\STATE Let $\mbox{Subres}_1(P,Q)(x,y,t)=\mathbf{sres}_{1}(x,y)\, t+\mathrm{sr}(x,y).$
\STATE Substitute $x=x(t,\alpha)$, $y=y(t,\alpha)$ in   $\mathbf{sres}_{1}(x,y)$.
\STATE Find the real solutions for $t$ of $\mathbf{sres}_{1}(\varphi(t,\alpha))=0$, where $\alpha^2=\widehat{U}^2(t)+\widehat{V}^2(t)$. For this purpose:
\begin{itemize}
\item [a.] Compute the polynomial $\xi_1(t)\,\alpha+\eta_1(t)$ by substituting $\alpha^2=\widehat{U}^2(t)+\widehat{V}^2(t)$ in $\mathbf{sres}_1(x(t,\alpha),y(t,\alpha))$, and keeping the numerator.
\item [b.] Compute the polynomial $\xi_1^2(t)\, (\widehat{U}^2(t)+\widehat{V}^2(t))-\eta_1^2(t)$,  and $\omega(t)$. 
\item [c.] Find the real roots of $\omega(t)$.
\end{itemize}
\STATE  Return the list ${\mathcal B}$ of the real roots of $\omega(t)$. 
\end{algorithmic}
\end{algorithm}

Notice that squaring in $\xi_1(t)\cdot\alpha=-\eta_1(t)$, which results from Equation \eqref{uno}, does not introduce fake solutions, because we are interested both in the case $\alpha=\sqrt{\widehat{U}^2(t)+\widehat{V}^2(t)}$ and $\alpha=-\sqrt{\widehat{U}^2(t)+\widehat{V}^2(t)}$; in fact, both correspond to offset points, one belonging to the exterior offset and the other one belonging to the interior offset. The only superfluous $t$ values that can appear when computing the real roots of $\omega(t)$ are a consequence of Remark \ref{conv-false}. These superfluous values appear when $H(x,y)=\mbox{Res}_t(P(x,y,t),Q(x,y,t))$ has extraneous factors that intersect ${\mathcal O}_d({\mathcal C})$, and correspond to the $t$-values generating these intersection points. Note also that if $\phi(t)$ is polynomial, then there are no extraneous factors, and therefore no superfluous $t$-values appear. 

In order to identify superfluous $t$ values, we first recall that the conditions for the appearance of extraneous factors, in the case of a non-polynomial $\phi(t)$, are described in Subsection \ref{subsec-off}. If these conditions hold, then the extraneous factors can be computed \cite{Far2}, and any  superfluous $t$ value must give rise to a point on one of these extraneous factors. So in order to check if a real root $t_i$ of $\omega(t)$ is superfluous, one must first check if $t_i$ gives rise to a point on some extraneous factor. In the affirmative case, $t_i$ could still be non-superfluous whenever it generates a singularity (local or self-intersection) of ${\mathcal O}_d({\mathcal C})$. In order to detect this last situation, one can directly test: (a) if there exists another $t_j\neq t_i$, $t_j$ a real root of $\omega(t)$, such that $\phi_d(t_i)=\phi_d(t_j)$ (in which case $t_j$ and $t_i$ give rise to a self-intersection of ${\mathcal O}_d({\mathcal C})$); (b) if $\mbox{lim}_{t\to t_i}\phi_d'(t)=0$, in which case $t=t_i$ is a local singularity of ${\mathcal O}_d({\mathcal C})$. In practice, and unless $t_i$ is rational, (a) or (b) can be tested only up to a certain tolerance. 

\begin{example}
Let ${\mathcal C}$ be the \emph{cardioid}, parametrized by 
$$
\mathcal{X}(t) = {\frac {-1024\,{t}^{3}}{256\,{t}^{4}+32\,{t}^{2}+1}}
,\,\,\,\,\mathcal{Y}(t)={\frac {-2048\,{t}^{4}+128\,{t}^{2}}{256\,{t}^{4}+32\,{t}^{2}+1}}.
$$
Let us now determine the singularities of the offset ${\mathcal O}_d({\mathcal C})$, for $d=1$. 
We first observe that $\mathrm{gcd}(X(t),Y(t),W(t))=1$, and that the parametrization is proper. Moreover the polynomials $U$ and $V$ are 
$$
U(t)=1024\,{t}^{2} \left( 16\,{t}^{2}-3 \right)  \left( 16\,{t}^{2}+1 \right),
$$
$$
V(t)=-256\,t \left( 48\,{t}^{2}-1 \right)  \left( 16\,{t}^{2}+1 \right). 
$$
The sum $U^2(t)+V^2(t)=65536\,{t}^{2} \left( 16\,{t}^{2}+1 \right) ^{5}$ is not a perfect square, and therefore ${\mathcal O}_d({\mathcal C})$ is irreducible for $d=1$. Additionally, $\gcd(U(t),V(t))= 256\,t \left( 16\,{t}^{2}+1 \right)$; hence, we define $ {\widehat{U}}(t)  = 4\,t \left( 16\,{t}^{2}-3 \right)$  and $\widehat{V}(t) =-48\,{t}^{2}+1$. Now, the polynomials $P(x,y,t)$ and $Q(x,y,t)$ are 
 $$P(x,y,t)=64\,x{t}^{3}- \left( 128+48\,y \right) {t}^{2}-12\,tx+y,$$
$$Q(x,y,t)= 256  \left( {x}^{2}+{y}^{2}+16\,y+63 \right) {t}^{4}+2048\,x{t}^{3}+ 32
 \left({x}^{2}+ {y}^{2}-8\,y -1\right) {t}^{2}+ ({x}^{2}+ {y}^{2}-1).
$$

The principal subresultant $\mathbf{sres_1} (x,y)$ is equal to
\begin{eqnarray*}
 \mathbf{sres_1} (x,y)& = & 1764\,x-5\,{x}^{7}+218\,{x}^{5}+903\,{x}^{3}-76\,{x}^{5}y-15\,{x}^{5}{
y}^{2}  \\
  &   &   -152\,{x}^{3}{y}^{3}-15\,{x}^{3}{y}^{4}-76\,{y}^{5}x-5\,{y}^{6}x
+2596\,y{x}^{3}+2020\,x{y}^{3}\\
& & +68\,{x}^{3}{y}^{2}-150\,{y}^{4}x+10024
\,yx+8823\,x{y}^{2}
\end{eqnarray*}
After running the algorithm {\tt OffsetSing}, the degree in $t$ of the polynomial $\tilde{\omega}(t)$ is equal to 29. By computing its square free part and dividing out the common factors with $W(t)=256\,{t}^{4}+32\,{t}^{2}+1$, we get
$$
\omega(t) = \left( {t}^{4}+{\frac {113}{9800}}{t}^
{2}+{\frac {1}{12544}} \right)   
 \left( {t}^{2}-{\frac {9}{3952}} \right)  \left( {t}^{6}-{\frac {3}{
3952}}{t}^{4}+{\frac {5}{63232}}{t}^{2}-{\frac {1}{1011712}}
 \right).
$$
Finally, by approximating the real roots of $\omega(t)$, we get ${\mathcal B}=\{t_1,t_2,t_3,t_4\}$ where
$$
t_1= -0.04772 , t_2=0.04772  ,t_3=-0.08699 ,  t_4=0.08699.
$$
The offset has a self-intersection at the points generated by $t_3$ and $t_4$, and two local singularities at the points generated by $t_1$ and $t_2$ (see Figure \ref{figura_cardioide}). Notice that ${\mathcal C}$ has a cusp, i.e. a local singularity, at the point $(0,0)$, generated by $t=0$. However, $t=0$ does not belong to ${\mathcal B}$. This is certainly not contradictory with Theorem \ref{bigth}, because one can check that the point $(0,0)$ does not generate any singularity of ${\mathcal O}_d({\mathcal C})$.

\begin{figure}
\begin{center}
\includegraphics[scale=0.25]{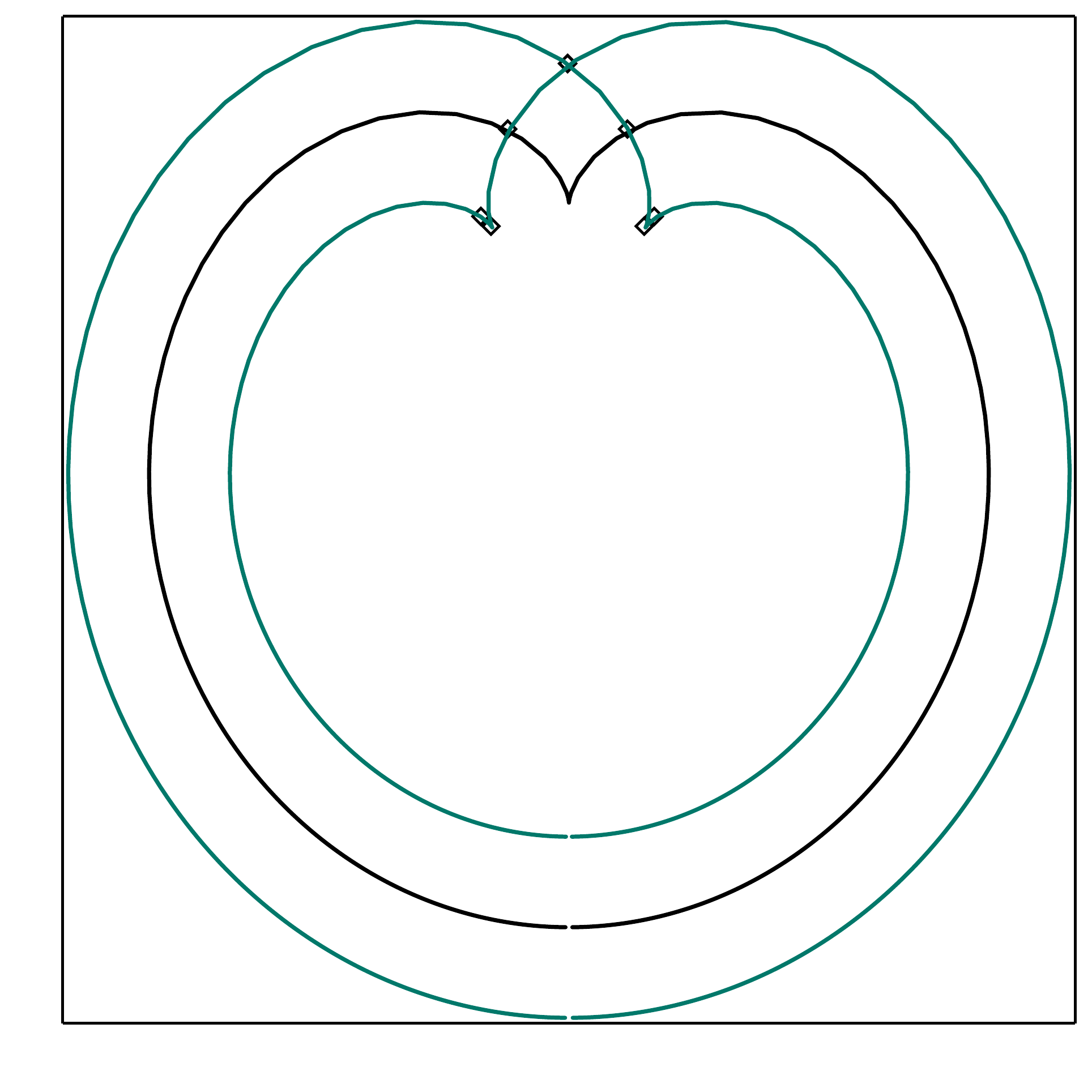}
\end{center}
\caption{Cardioid curve}\label{figura_cardioide}
\end{figure}

\end{example}

\subsection{Growing of coefficients and degrees, and complexity of the algorithm.}

In this subsection we will analyze the growing of coefficients and degrees in Algorithm {\tt OffsetSing}, as well as the bit complexity of Algorithm {\tt OffsetSing}. Notice that the bit complexity takes into account the growing of coefficients in the algorithm. We will use the standard \emph{Big Oh} notation $\OOO$, usually employed in complexity analysis, as well as the \emph{Soft Oh} notation $\tOOO$, where logarithmic factors are ignored. Furthermore, in order to simplify the analysis we will assume that the parametrization has integer coefficients.

We need to fix some notation: we denote by $k$ the maximum of the degrees of $X(t),Y(t),W(t)$, and we denote by $\tau$ the maximum of the bitsizes of the coefficients of $X(t),Y(t),W(t)$. Recall that if the bitsize of the coefficients of a polynomial is $\tau$, then the coefficients of the polynomial are bounded by $2^{\tau}$. Furthermore, in our analysis we will assume that $\tau>>\mbox{log}_2(k)$, so that $\mbox{log}_2(k)$ can be neglected when compared to $\tau$. This assumption corresponds to the case of highest computational cost. Notice that:

\begin{itemize}
\item [(a)] If we multiply two polynomials of degree $k$ with coefficients bounded by $2^{\tau_1}$ and $2^{\tau_2}$, the coefficients of the resulting polynomial are bounded by $k2^{\tau_1+\tau_2}$. Hence, the bitsize of the product is bounded by $\tau_1+\tau_2+\mbox{log}_2(k)$. Assuming that $\tau_1,\tau_2>>\mbox{log}_2(k)$ the bitsize of the product of two polynomials, using the notation fixed at the beginning of the subsection, is $\OOO(\tau_1+\tau_2)$. If $\tau_1=\tau_2=\tau$ then we get $\OOO(\tau)$. 
\item [(b)] If we add $r$ polynomials of degree $k$ with coefficients bounded by $2^\tau$, the coefficients of the sum are bounded by $r2^{\tau}$, so the bitsize of the coefficients of the sum is bounded by $\tau+\mbox{log}_2 r$. 
\item [(c)] By repeatedly applying (a), multiplying $s$ polynomials of degree $k$ with coefficients of bitsize $\tau$ yields a polynomial where the bitsize of the coefficients is $\OOO(s\tau)$. 
\end{itemize}

We also recall the following results on the complexity of some basic algorithms. We acknowledge here the help of Michael Sagraloff for pointing out some references and hints about the bit complexity of basic operations. 

\begin{itemize}
\item The bit complexity of multiplying two univariate polynomials of degree $n$ with coefficients of bitsize $\mu$ is $\tOOO(n\mu)$ (see Corollary 8.27 of \cite{vonzurger}). 
\item By repeatedly squaring and taking into account the complexity of multiplying univariate polynomials, the computation of the $p$-th power of a polynomial $f(t)$ of degree $n$ with coefficients of bitsize $\mu$ can be done in $\tOOO(n\mu p^2)$ time. 
\item In order to multiply two polynomials $F_1(t,s)=sf_1(t)+g_1(t)$ and $F_2(t,s)=sf_2(t)+g_2(t)$, we can evaluate $F_1,F_2$ at three random values of $s$, multiply the corresponding univariate polynomials, and then recover the value of $F_1(t,s)\cdot F_2(t,s)$ by interpolation. Hence if $f_i(t),g_i(t)$ have degrees bounded by $n$ and coefficients of bitsize $\mu$, then the cost is dominated by the cost of univariate polynomial multiplication, $\tOOO(n\mu)$. In order to compute $(sf_i(t)+g_i(t))^p$ we can apply a similar strategy. Since we need $p+1$ evaluations of $s$, the bit complexity of computing $F_1(t,s)\cdot F_2(t,s)$ is the result of multiplying $p+1$ times the bit complexity of computing the power of a univariate polynomial; hence we get $\tOOO(n\mu p^3)$. 
\item In order to compute the product of two bivariate polynomials $F(t,s)$ and $G(t,s)$ of degree $n$ and bitsize $\mu$, we need to evaluate $s$ at $2n+1$ points $s_1,\ldots,s_{2n+1}$, compute the products $F(t,s_i)\cdot G(t,s_i)$, and then recover $F(t,s)\cdot G(t,s)$ by interpolation. Now $F(t,s_i)$ and $G(t,s_i)$ are univariate polynomials of degree $n$ and bitsize $\tOOO(n+\mu)$. Hence computing $F(t,s_i)\cdot G(t,s_i)$ has a cost of $\tOOO(n^2+n\mu)$. We need to carry out this process $2n+1$ times, so the total cost is $\tOOO(n^3+n^2\mu)$. The interpolation part is dominated by the total cost of the multiplications. Therefore the bit complexity of computing $F(t,s)\cdot G(t,s)$ is $\tOOO(n^3+n^2\mu)$; the coefficients of $F(t,s)\cdot G(t,s)$ have bitsize $\OOO(\mu)$. 
\end{itemize}

Now let us analyze each step of Algorithm {\tt OffsetSing}.  
\vspace{3 mm}

\noindent {\it Step 1.} The degrees of $U(t),V(t)$ are bounded by $2k-1$, and the bitsizes of $U(t),V(t)$ are, according to the above observations, $\OOO(\tau)$. Hence the degrees of $\tilde{P}(x,y,t)$ and $\tilde{Q}(x,y,t)$ are bounded by $3k-1$, and their bitsizes are also $\OOO(\tau)$. In order to compute $P(x,y,t)$ and $Q(x,y,t)$ we need to remove the $t$-contents of $\tilde{P}(x,y,t)$ and $\tilde{Q}(x,y,t)$. By Lemma 11 in \cite{Sagr}, a divisor of a polynomial in $  {\Bbb Z}[x_1,\ldots,x_\ell]$ of degree $N$ with coefficients of bitsize bounded by $\mu$ has coefficients of bitsize $\tOOO(N+\mu)$. Since $P(x,y,t)$ and $Q(x,y,t)$ are divisors of $\tilde{P}(x,y,t)$ and $\tilde{Q}(x,y,t)$, in our case $N=3k-1$ and $\mu=\OOO(\tau)$; hence, $P(x,y,t)$ and $Q(x,y,t)$ have coefficients of bitsize $\tOOO(k+\tau)$. Furthermore, the degrees of $P(x,y,t)$ and $Q(x,y,t)$ are bounded by $3k-1$. So we finish this step with polynomials $P,Q$ of degrees $\OOO(k)$, and bitsizes $\tOOO(k+\tau)$.

The bit complexity of step 1 is dominated by the computation of the $t$-contents of $\tilde{P},\tilde{Q}$. From \eqref{cont1}, \eqref{cont2}, \eqref{cont3}, the $t$-contents of $\tilde{P},\tilde{Q}$ are products of gcds of univariate polynomials of degree $\OOO(k)$ and bitsize $\tOOO(\tau)$. Hence the bit complexity of step 1 is $\tOOO(k^2\tau)$ (see Section 11.2 in \cite{vonzurger}). 

\vspace{3 mm}

\noindent {\it Step 2.} Let us analyze the computation of $\mbox{Subres}_1(P,Q)(x,y,t)$. According to \cite{Diochnos}, if $f,g\in {\Bbb Z}[y_1,\ldots,y_\ell][t]$, $\mbox{deg}_t(f)=p$, $\mbox{deg}_t(g)=q$, $\mbox{deg}_{y_i}(f)\leq e_i$, $\mbox{deg}_{y_i}(g)\leq e_i$, $e=e_1\cdots e_\ell$, and the bitsizes of $f,g$ are bounded by $\mu$, then:
\begin{itemize}
\item [(i)] The degree in $y_i$ of each subresultant of $f,g$ with respect to $t$ is bounded by $(p+q)e_i$.
\item [(ii)] The complexity of the computation of any element of the subresultant sequence is $\tOOO(q(p+q)^{\ell+1} e \mu)$.
\item [(iii)] The total degree in $y_1,\ldots,y_\ell$ of each subresultant is bounded by $(p+q)\sum_{i=1}^\ell e_i$.
\item [(iv)] The bitsize of the coefficients of the subresultants is bounded by $(p+q)\mu$.
\end{itemize}
In our case, $f:=P$, $g:=Q$, $\ell=2$, $y_1:=x$, $y_2:=y$, $p=q=\OOO(k)$. Furthermore, since $\mbox{deg}_x(P)=\mbox{deg}_y(P)=1$, $\mbox{deg}_x(Q)=\mbox{deg}_y(Q)=2$, then $e_1=2$, $e_2=2$, and therefore $e=4$. Also, $\mu=\tOOO(k+\tau)$. Hence, 
\begin{itemize}
\item [(i)] The degrees in $x$ or $y$ of each subresultant of $P,Q$ with respect to $t$ is $\OOO(k)$. 
\item [(ii)] The complexity of the computation of any element of the subresultant sequence is $\tOOO(k^5+k^4\tau)$.
\item [(iii)] The total degree in $x,y$ of each subresultant is $\OOO(k)$.
\item [(iv)] The bitsize of the coefficients of the subresultants is $\tOOO(k^2+k\tau)$.
\end{itemize}

Therefore the bit complexity of step 2 is $\tOOO(k^5+k^4\tau)$.

\vspace{3 mm}

\noindent {\it Step 3.} Let $\mbox{Subres}_1(P,Q)(t)=\mathbf{sres}_1(x,y)\,t+\mbox{sr}(x,y)$. According to the previous analysis, $\mathbf{sres}_1(x,y)$ is a bivariate polynomial of degree $N$ where $N=\OOO(k)$. Therefore the number of terms of $\mathbf{sres}_1(x,y)$ is bounded by ${N+2 \choose 2}=\OOO(k^2).$ Now we need to substitute $x:=x(t,\alpha)$ and $y:=y(t,\alpha)$ in $\mathbf{sres}_1(x,y)$. Notice that 
\[
\begin{array}{lcr}
x(t,\alpha)= & \displaystyle{\frac{X(t)}{W(t)}+d\frac{\widehat{V}(t)}{\alpha}=} & \displaystyle{\frac{\alpha X(t)+d\widehat{V}(t)W(t)}{\alpha W(t)}},\\
y(t,\alpha)= & \displaystyle{\frac{Y(t)}{W(t)}-d\frac{\widehat{U}(t)}{\alpha}=} & \displaystyle{\frac{\alpha Y(t)-d\widehat{U}(t)W(t)}{\alpha W(t)}}.
\end{array}
\]
Writing \[\mathbf{sres}_1(x,y)=\sum_{\begin{array}{c}0\leq i,j\leq N\\ 0\leq i+j\leq N\end{array}}a_{ij}x^i y^j,\]we have that $\mathbf{sres}_1(x(t,\alpha),y(t,\alpha))$ is equal to  
\begin{equation}\label{sum}
\sum_{\begin{array}{c}0\leq i,j\leq N\\ 0\leq i+j\leq N\end{array}}a_{ij} \displaystyle{\frac{(\alpha X(t)+d\widehat{V}(t)W(t))^i\cdot (\alpha Y(t)-d\widehat{U}(t)W(t))^{j}\cdot [\alpha W(t)]^{N-(i+j)}}{\alpha^N W^N(t)}}.
\end{equation}
The total degrees, as polynomials in $\alpha,t$ of $\alpha X(t)+d\widehat{V}(t)W(t)$, $\alpha Y(t)-d\widehat{U}(t)W(t)$ are bounded by $3k-1$, and the total degree of $\alpha W(t)$ is bounded by $k+1$. Hence the total degree of each term of the numerator of \eqref{sum} is bounded by 
\begin{equation}\label{nume}
(3k-1)(i+j)+(k+1)(N-i-j)=(2k-2)(i+j)+(k+1)N,
\end{equation}
where $0\leq i+j\leq N$. Thus, the total degree of $\eqref{sum}$ is bounded by the result of replacing $i+j=N$ in \eqref{nume}. Since $N=\OOO(k)$ we get that the total degree of the numerator of \eqref{sum} is $\OOO(k^2)$. Notice also that the degree in $\alpha$ of the numerator of $\eqref{sum}$ is bounded by $N=\OOO(k)$. 

Since $\widehat{U}(t), \widehat{V}(t)$ are factors of $U(t),V(t)$ and the bitsize of the coefficients of $U(t),V(t)$ is $\OOO(\tau)$ and their degrees are $\OOO(k)$, then the bitsize of $\widehat{U}(t), \widehat{V}(t)$ is $\tOOO(k+\tau)$. Then the bitsize of the coefficients of $(\alpha X(t)+d\widehat{V}(t)W(t))^i$ is $\tOOO(i(k+\tau))$, the bitsize of the coefficients of $(\alpha Y(t)-d\widehat{U}(t)W(t))^{j}$ is $\tOOO(j(k+\tau))$, and the bitsize of the coefficients of $[\alpha W(t)]^{N-(i+j)}$ is bounded by $\tOOO((N-i-j)\tau)$. For each $i,j$ we have $i+j+[N-(i+j)]=N$, and therefore the bitsize of the coefficients of $(\alpha X(t)+d\widehat{V}(t)W(t))^i\cdot (\alpha Y(t)-d\widehat{U}(t)W(t))^{j}\cdot [\alpha W(t)]^{N-(i+j)}$ is $\tOOO(N(k+\tau))=\tOOO(k^2+k\tau)$. Since the bitsize of $a_{ij}$ is $\tOOO(k^2+k\tau)$, the multiplication by $a_{ij}$ yields coefficients of bitsize $\tOOO((k^2+k\tau)^2)$. We may need to add $\OOO(k^2)$ coefficients of this bitsize, but neglecting logarithmic terms we again get bitsize $\tOOO((k^2+k\tau)^2)$.  

Let us compute now the complexity of step 3. We observe that each $(\alpha X(t)+d\widehat{V}(t)W(t))^i$ and $(\alpha Y(t)-d\widehat{U}(t)W(t))^{j}$ in \eqref{nume} has degree $\OOO(k^2)$ and bitsize $\tOOO(k^2+k\tau)$. Therefore 
from the results at the beginning of the subsection, the complexity of computing each multiplication \[(\alpha X(t)+d\widehat{V}(t)W(t))^i\cdot (\alpha Y(t)-d\widehat{U}(t)W(t))^{j}\cdot [\alpha W(t)]^{N-(i+j)}\] in \eqref{nume} is dominated by $\tOOO\left((k^2)^3+(k^2)^2(k^2+k\tau)\right)=\tOOO(k^6+k^5\tau)$. Since $\mathbf{sres}_1(x,y)$ has $\OOO(k^2)$ terms, we need to perform $\OOO(k^2)$ multiplications of this kind. Hence we get a complexity $\tOOO(k^2\cdot (k^6+k^5\tau))=\tOOO(k^8+k^7\tau)$ for step 3.

\vspace{3 mm}

\noindent {\it Step 4.} Let $a(t,\alpha)$ be the numerator of $\mathbf{sres}_1(x(t,\alpha),y(t,\alpha),t)$, i.e.
\[a(t,\alpha)=a_N(t)\alpha^N+a_{N-1}(t)\alpha^{N-1}+\cdots + a_0(t).\]
Recall that the total degree of $a(t,\alpha)$ is $\OOO(k^2)$, and the degree in $\alpha$ of $a(t,\alpha)$ is $\OOO(k)$. Now we need to substitute $\alpha^2:=b(t)=\widehat{U}^2(t)+\widehat{V}^2(t)$ into $a(t,\alpha)$. The degree of $b(t)$ is $\OOO(k)$, and the bitsize of the coefficients is $\OOO(k+\tau)$. Hence the computation of each power $(\widehat{U}^2(t)+\widehat{V}^2(t))^j$ has complexity $\tOOO((k^2+k\tau) j^2)$. Since $j\leq k$, this complexity is dominated by $\tOOO(k^4+k^3\tau)$. Furthermore, the bitsize of the coefficients of each $(\widehat{U}^2(t)+\widehat{V}^2(t))^j$ is dominated by $\tOOO(k^2+k\tau)$, and the degree in $t$ of $(\widehat{U}^2(t)+\widehat{V}^2(t))^j$ is dominated by $\OOO(k^2)$. The degree in $t$ of each $a_j(t)$ is $\OOO(k^2)$, and the bitsize of its coefficients is $\tOOO((k^2+k\tau)^2)$. Hence the complexity of each multiplication $a_j(t)(\widehat{U}^2(t)+\widehat{V}^2(t))^{j^{\star}}$ is dominated by $\tOOO(k^2\cdot (k^2+k\tau)^2)=\tOOO(k^4(k+\tau)^2)$, and gives rise to coefficients of bitsize $\tOOO((k^2+k\tau)^2)$. We need to compute $\OOO(k)$ products of this kind, so we get a total complexity of $\tOOO(k\cdot k^4(k+\tau)^2)=\tOOO(k^5(k+\tau)^2)$. The degree in $t$ that we get is $\OOO(k^2)$.

Finally we arrive at equation \eqref{uno} and then, after squaring, we compute the polynomial $\tilde{\omega}(t)$ in \eqref{pol}, and then $\omega(t)$. These operations do not modify the total complexity of step 4. Furthermore, the degree of $\omega(t)$ is $\OOO(k^2)$, and the bitsize of the coefficients of $\omega(t)$ is $\tOOO((k^2+k\tau)^2)$. According to Theorem 5 in \cite{Sagr}, the cost of isolating the roots of $\omega(t)$ is $\tOOO(k^6(k+\tau)^2)$. Therefore, the overall bit complexity of step 4 is $\tOOO(k^6(k+\tau)^2)$.
 
The computational cost of checking whether or not ${\mathcal M}$ is irreducible is very low, and does not modify the overall complexity of the algorithm. Adding the complexities for step 1, step 2, step 3 and step 4, we achieve an overall bit complexity of $\tOOO(k^6(k+\tau)^2)$ for Algorithm {\tt OffsetSing}. We summarize the previous reasonings in the following result.

\begin{theorem} \label{complex}
 The overall bit complexity of Algorithm {\tt OffsetSing} is bounded by $\tOOO(k^6(k+\tau)^2)$. The growing of the coefficients in the algorithm is dominated by $\tOOO((k^2+k\tau)^2)$, and the growing of the degrees is dominated by $\OOO(k^2)$. Furthermore, the polynomial $\omega(t)$ computed by Algorithm {\tt OffsetSing} has degree $\OOO(k^2)$ and coefficients of bitsize dominated by $\tOOO((k^2+k\tau)^2)$. 
\end{theorem}


\subsection{Experimental Results}

We have implemented our algorithm in the computer algebra system Maple 16. Since the algorithm requires to compute subresultants, we tested several algorithms for carrying out this operation (see \cite{AMC}); in practice, the best method seems to be  the direct computation of the subresultant chain using the command  {\tt RegularChains[ChainTools][SubresultantChain]} of Maple 16. Some of the rational curves we have tested, denoted by $\mathbf{C}_i$, $i\in \{1,\ldots, 8\}$, are enlisted below: 


\begin{eqnarray*}
\mathbf{C}_1& : & x={\frac {t+{t}^{3}}{1+{t}^{4}}} , y={\frac {t-{t}^{3}}{1+{t}^{4}}} \\
\\
\mathbf{C}_2 & : &x= {\frac {-7\,{t}^{4}+288\,{t}^{2}+256}{{t}^{4}+32\,{t}^{2}+256}} , y={\frac {-80\,{t}^{3}+256\,t}{{t}^{4}+32\,{t}^{2}+256}} \\
\\
\mathbf{C}_3& : & x={\frac {18\,{t}^{4}+21\,{t}^{3}-7\,t-2}{18\,{t}^{4}+48\,{t}^{3}+64\,{t
}^{2}+40\,t+9}}  , y={\frac {36\,{t}^{4}+84\,{t}^{3}+73\,{t}^{2}+28\,t+4}{18\,{t}^{4}+48\,{
t}^{3}+64\,{t}^{2}+40\,t+9}} \\ 
\\
\mathbf{C}_4 &: & x= 1- \left( 1+{t}^{2} \right) ^{-1},y=t-{\frac {t}{1+{t}^{2}}}\\
\\
\mathbf{C}_5& : & x= {\frac {-{t}^{4}-6\,{t}^{2}+3}{ \left( 1+{t}^{2} \right) ^{2}}}, y=8\,{\frac {{t}^{3}}{ \left( 1+{t}^{2} \right) ^{2}}}\\
\\ 
\mathbf{C}_6& : &   x=\frac{1-3t^2}{(1+t^2)^2}, y=\frac{(1-3t^2)t}{(1+t^2)^2}\\
\\ 
\mathbf{C}_7& : &  x=\frac{t^2-3}{1+t^2}, y=\frac{t(-t^2+3)}{1+t^2}\\
\\ 
\mathbf{C}_8& : & x=\frac{87-7t^4+22t^3-55t^2-94t}{-73-56t^4-62t^2+97t}, y=\frac{-82-4t^4-83t^3-10t^2+62t}{-73-56t^4-62t^2+97t}\\
\\
\end{eqnarray*}

The results of our experiments on the curves $\mathbf{C}_i$ are summarized in Table 1: here, the variable {\tt Time} shows the computing time (in seconds); $d$ is the offsetting distance; {\tt $n_p$} is the number of real $t$-values computed by the algorithm; $\delta_t$ is the degree of $\omega(t)$,  $\tau$ is the bitsize of $\omega(t)$, and $\delta_t(P), \delta_t(Q)$ are the degrees in $t$ of $P(x,y,t)$ and $Q(x,y,t)$. Additionally, we include the degree $\deg(F(x,y))$ of the implicit equation of the offset, in order to give an idea of the size of the problem in each case. All the computations have been carried out with the Computer Algebra System
Maple 16 in an iMac with an Intel Core i3 processor
with speed revving up to 3.06GHz. 
\begin{center}
\begin{tabular}{|c|l|l|c|c|c|c|c|c|} \hline
Example & Time & $d$ & $n_p$  &$\delta_t$ &  $\tau$ &$\delta_t(P), \delta_t(Q) $   &$\delta(F(x,y)) $\\
\hline 
\hline $\mathbf{C}_1$ & 0.140  & 1 & 10   &30 &  22  &6, 4  &  12 \\
\hline $\mathbf{C}_2$ & 0.122  & 1 &   9 &   21 &41  &4, 4   & 10 \\
\hline $\mathbf{C}_3$ & 6.167  & 1 &  26  & 222 & 510 & 10, 8    & 20 \\
\hline $\mathbf{C}_4$ & 0.051  & 1 &   4 & 22&   16 & 4, 4 & 8\\ 
\hline $\mathbf{C}_5$ & 0.068  & 1 &   8 & 20 &   23& 3, 6  & 8 \\
\hline $\mathbf{C}_5$ & 0.092  & 0.3 & 12  &  22 & 53 & 3, 6   & 8 \\
\hline $\mathbf{C}_6$ &  0.244  &  1 & 21  &81&  84 &6, 6 &14 \\
\hline $\mathbf{C}_7$ & 0.074   & 1 &  9  &29&26 & 5, 4  & 10 \\
\hline $\mathbf{C}_8$ &  6.948  & 0.5  & 12   &  228 &  927 & 10, 8   & 20 
\\
\hline
\end{tabular}
\end{center}\begin{center}
{\bf Table 1:}  Examples.
\end{center}

The pictures corresponding to the examples in Table  1 can be found in Figure \ref{figuras_ejemplos}; from left to right, the curves $\mathbf{C}_1$, $\mathbf{C}_2$, $\mathbf{C}_3$ appear in the first row, $\mathbf{C}_4$, $\mathbf{C}_5$ with $d=1$, $\mathbf{C}_5$ with $d=0.3$ are shown in the second row, and $\mathbf{C}_6$, $\mathbf{C}_7$, $\mathbf{C}_8$ are shown in the third row. In all the cases, with the exception of $\mathbf{C}_7$, we get exactly the set of $t$-values generating the singularities of the offset, without any extra values. In the case of $\mathbf{C}_7$, we get a superfluous value, namely $t=0$. This $t$-value generates the offset points $(-4,0)$ and $(-2,0)$. None of these points are singularities of the offset; however, the point $(-2,0)$ is a singularity of ${\mathcal H}=\mbox{Res}_t(P,Q)$, which in this case has extraneous components, and corresponds to the intersection of one of these spurious components with the offset to $\mathbf{C}_7$. 

Notice that in general $n_p$ is bigger than the number of singularities, since several $t$-values give rise to the same self-intersection. For instance, in the case of $\mathbf{C}_5$ with $d=1$, the algorithm provides $12$ $t$-values. However, the curve has 9 singularities; 6 of these singularities are local singularities, each one generated by a different $t$-value, and three singularities are self-intersections, each one generated by two different $t$-values. Therefore, we get $n_p=6\cdot 1+3\cdot 2=12$.


\begin{figure}
\begin{center}
\centerline{$\begin{array}{ccc}
\includegraphics[scale=0.25]{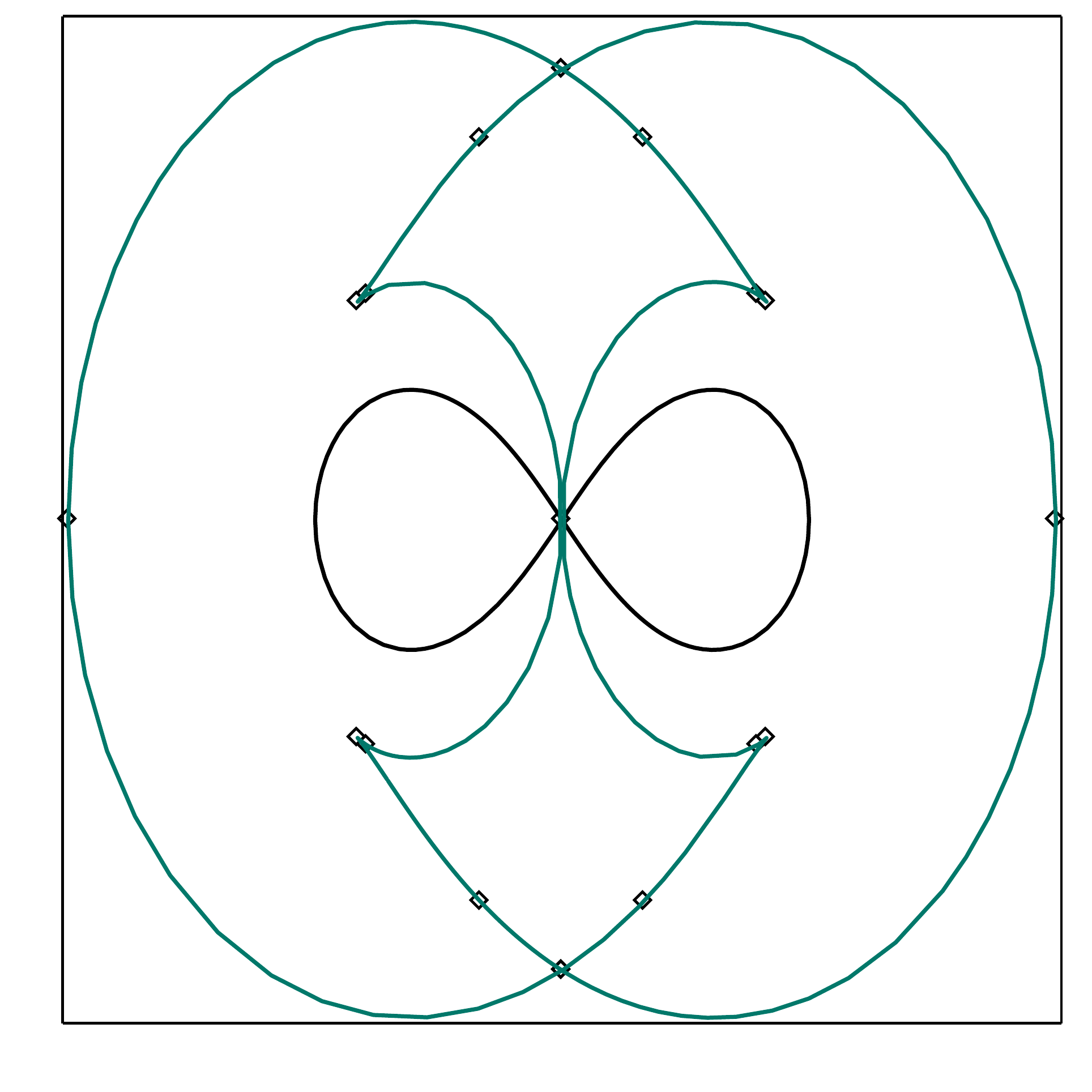}
 &\includegraphics[scale=0.25]{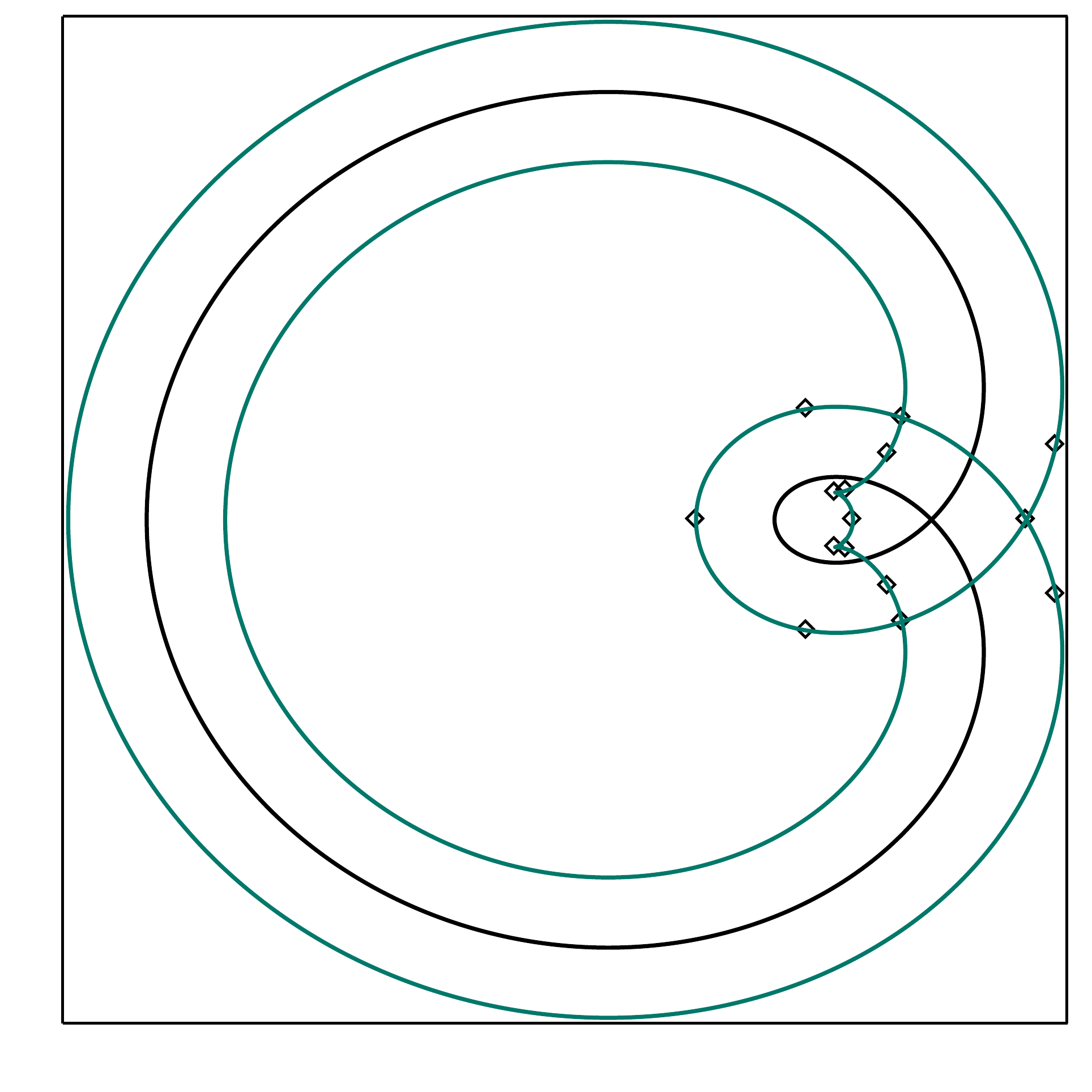}
 &\includegraphics[scale=0.25]{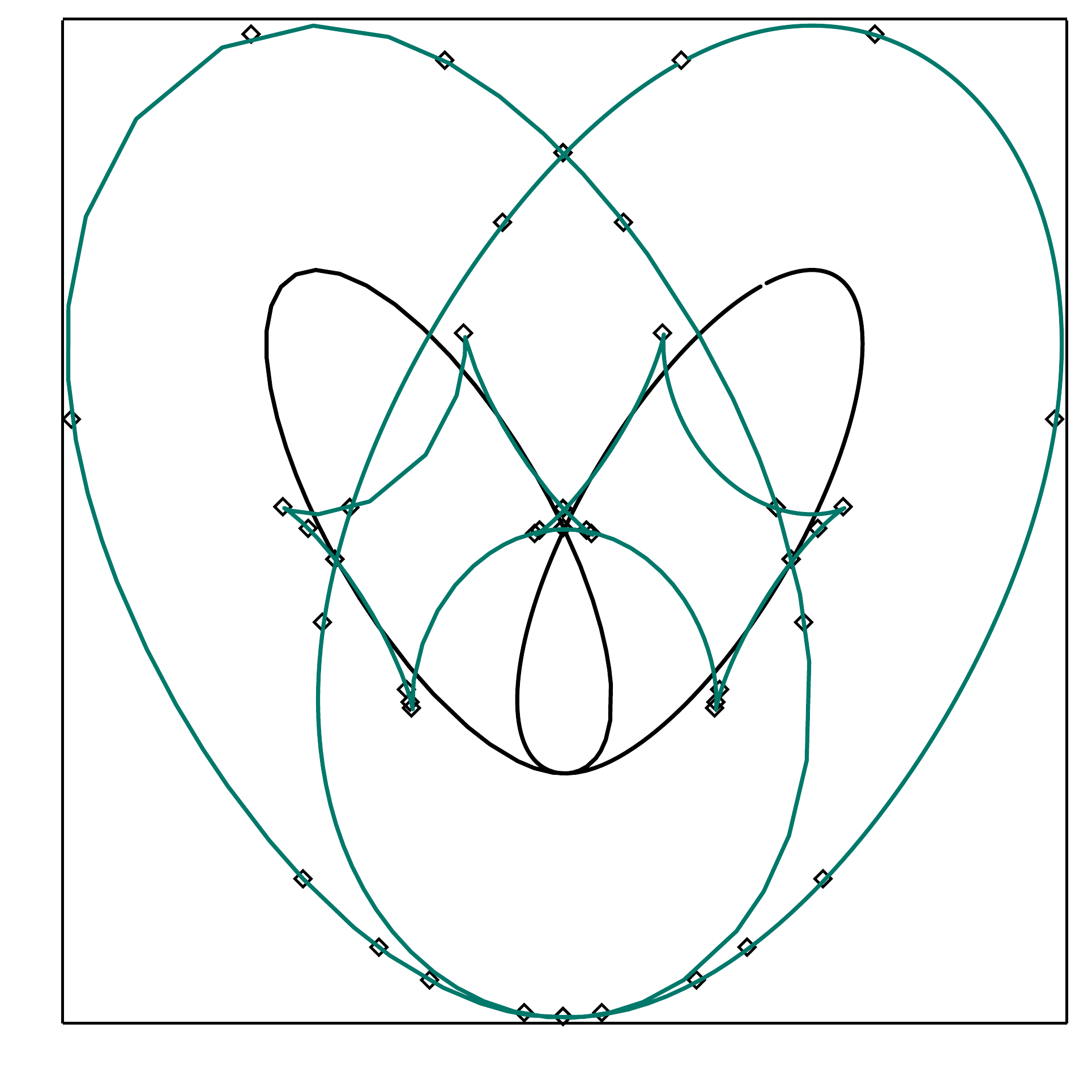} \\
\includegraphics[scale=0.25]{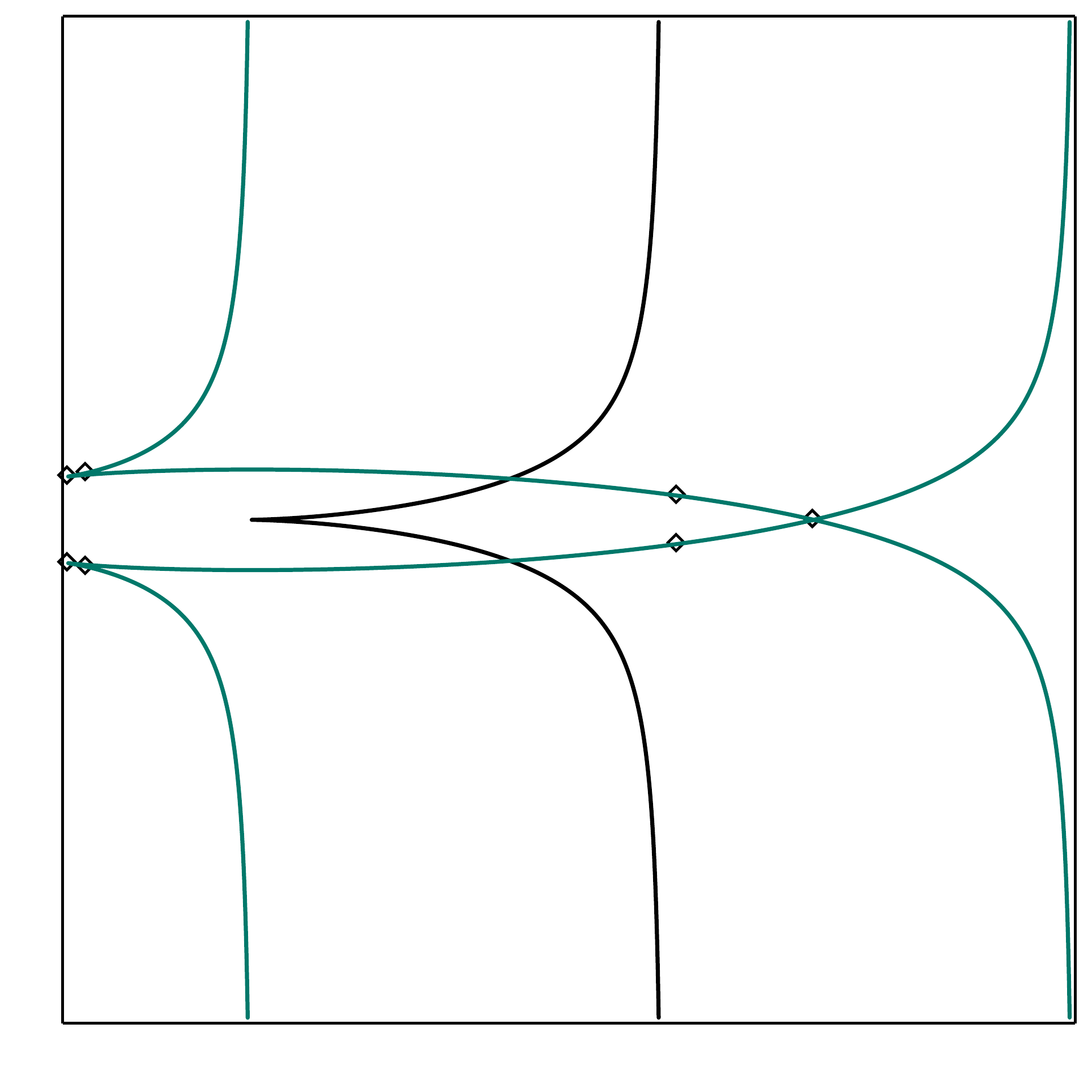} &
\includegraphics[scale=0.25]{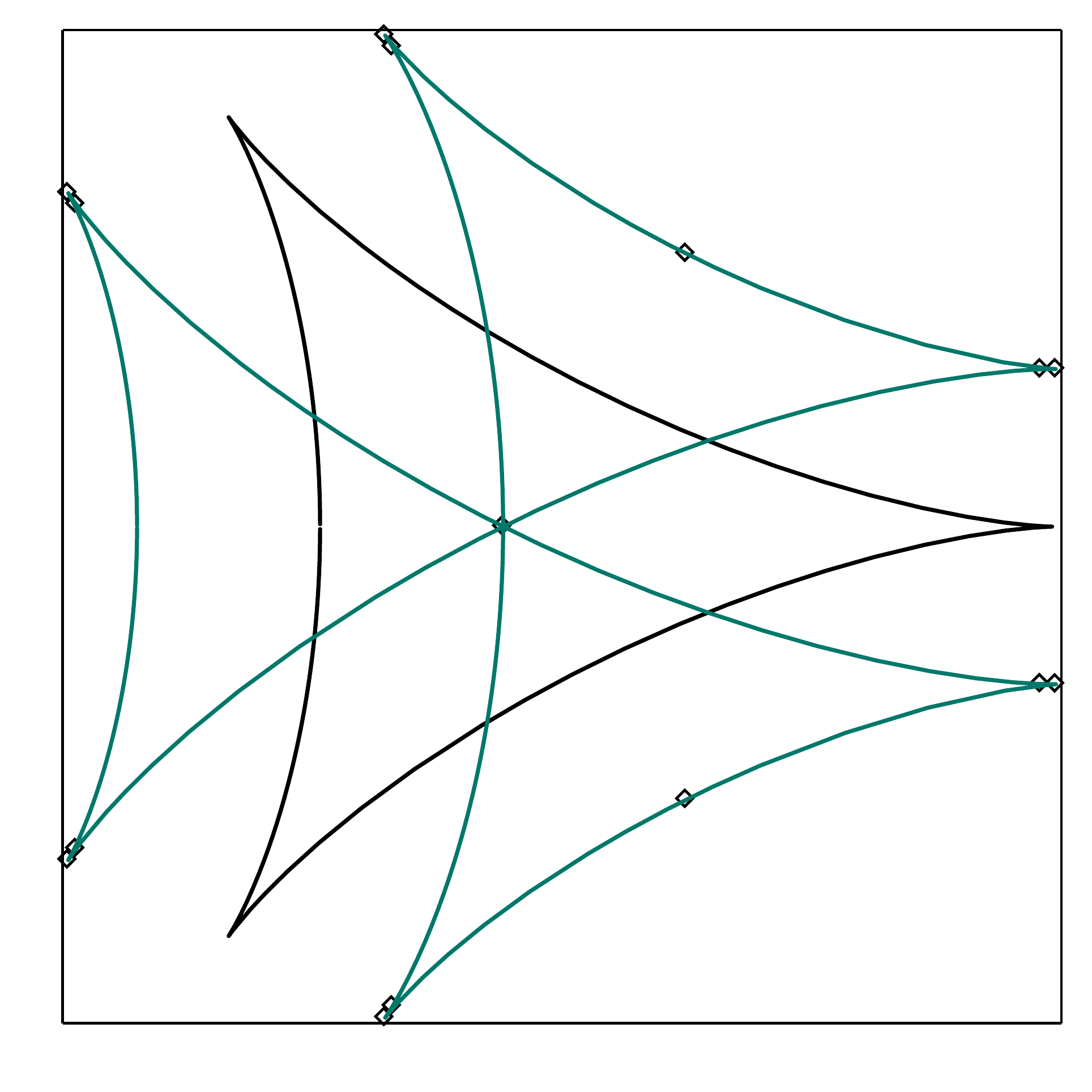} &
\includegraphics[scale=0.25]{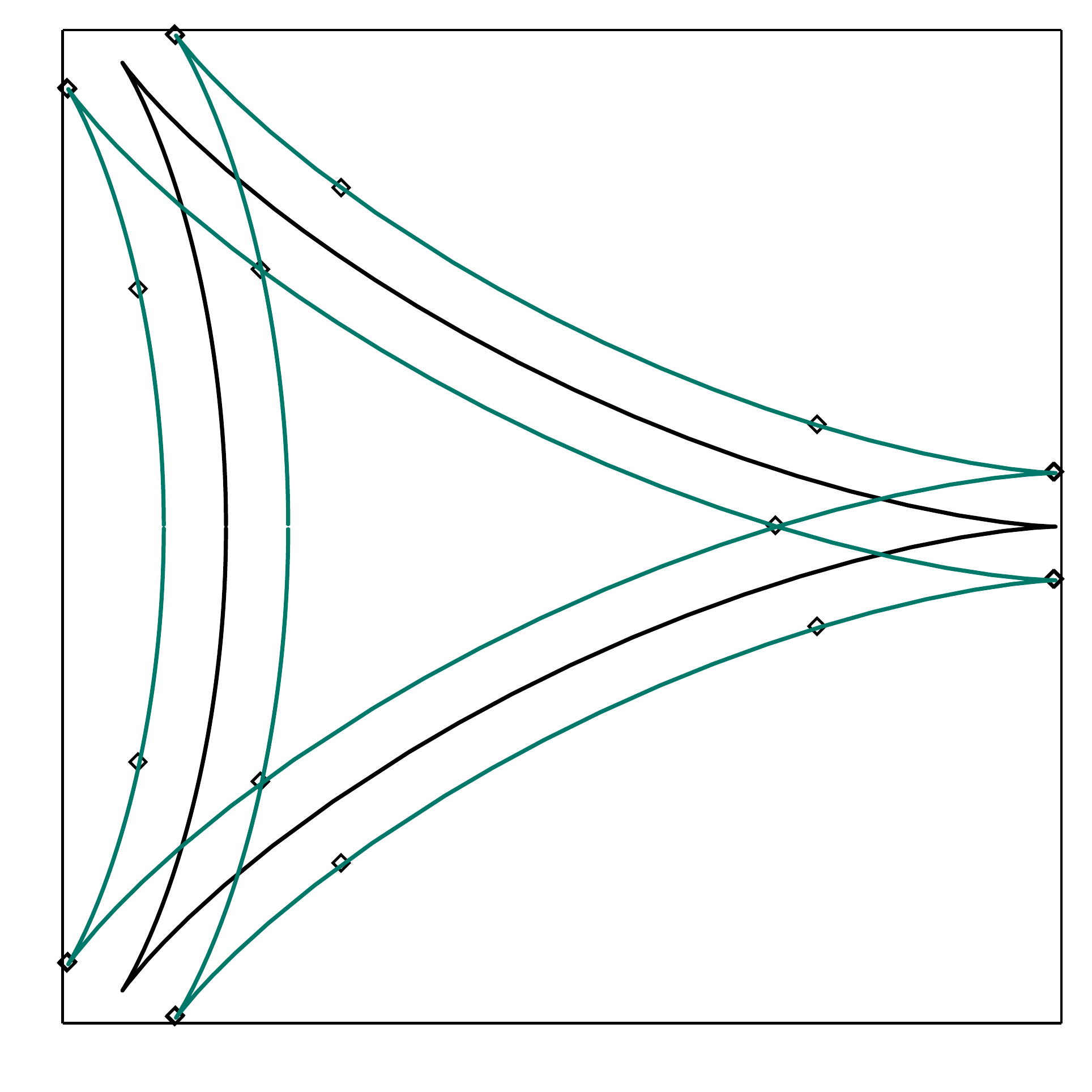}\\
\includegraphics[scale=0.25]{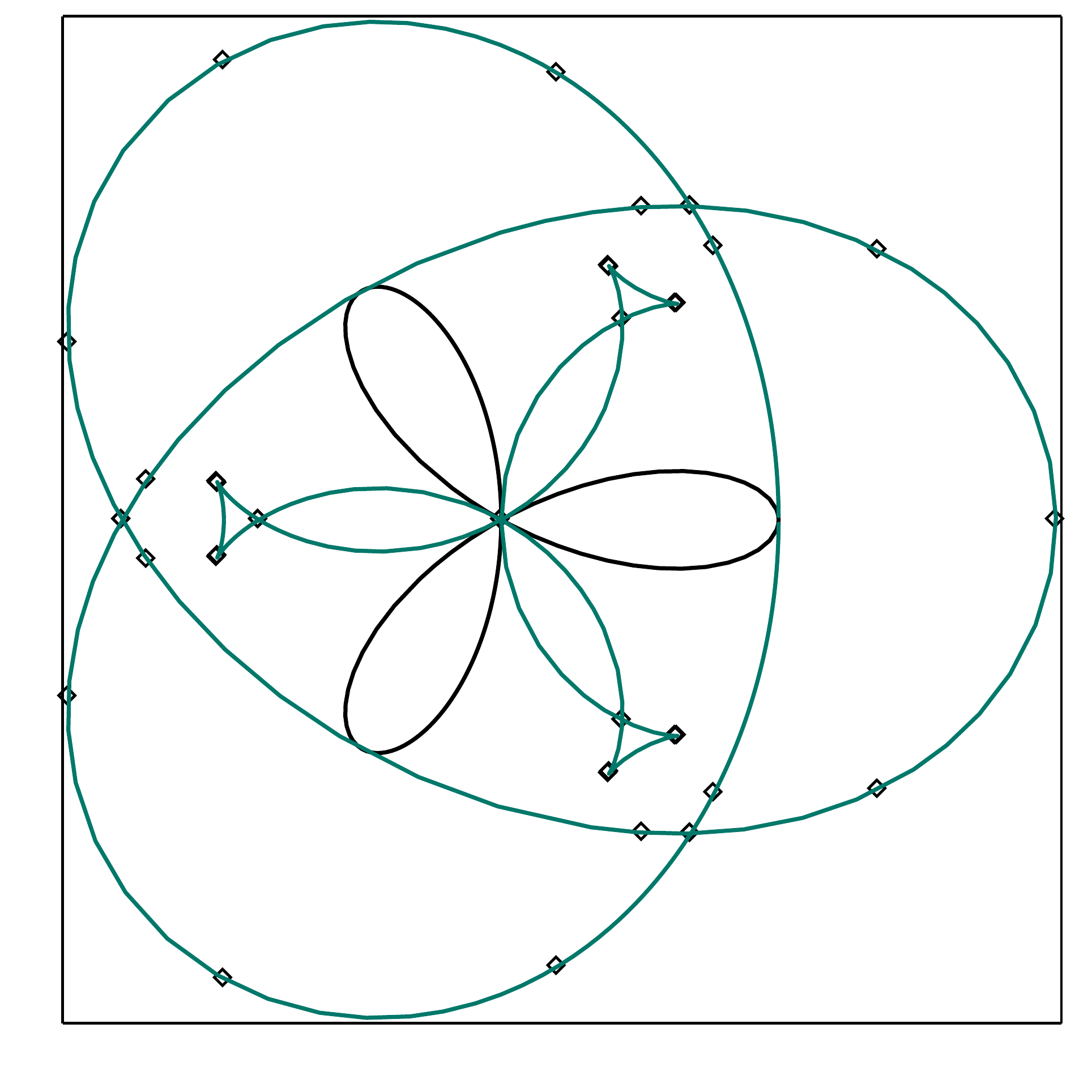}&\includegraphics[scale=0.25]{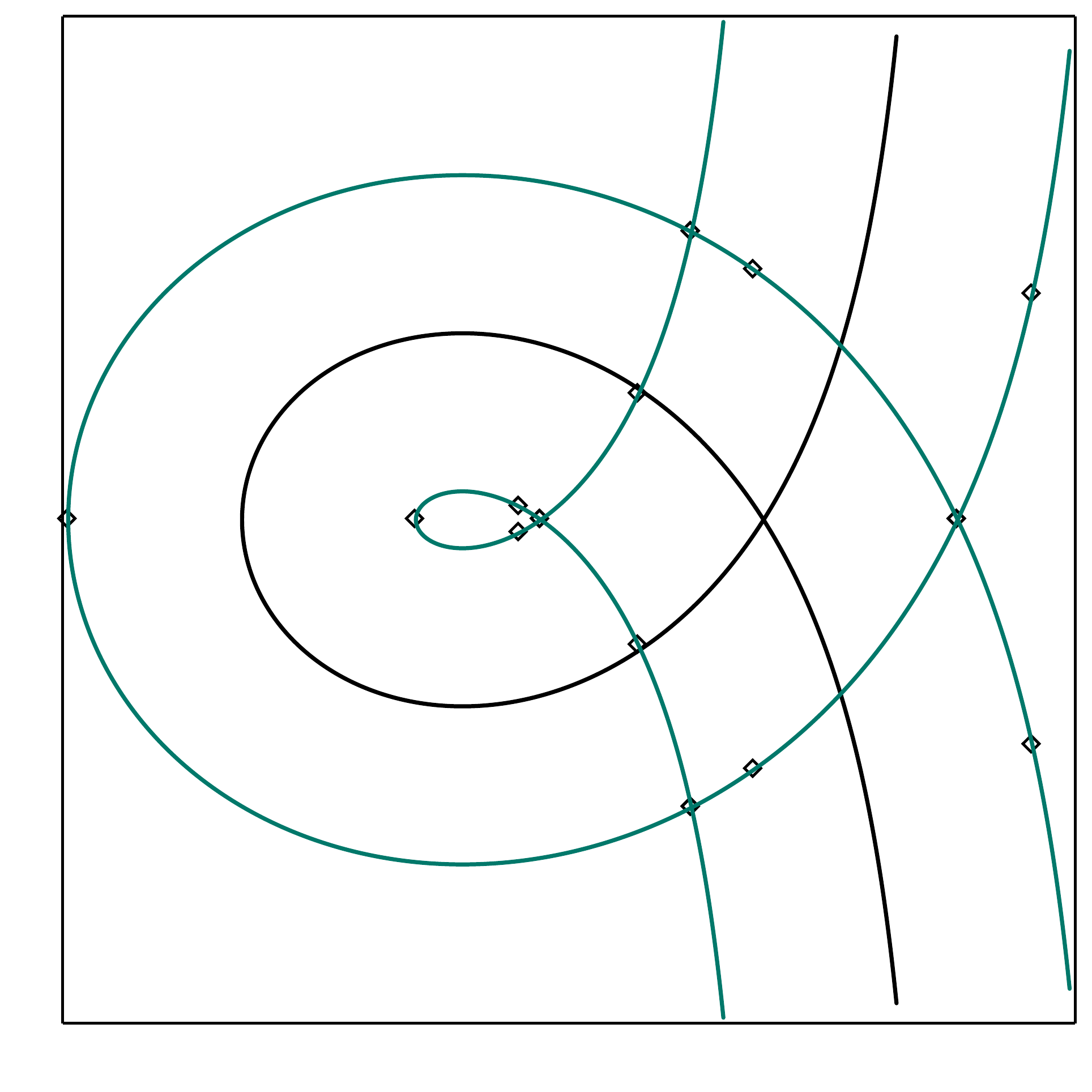} & \includegraphics[scale=0.25]{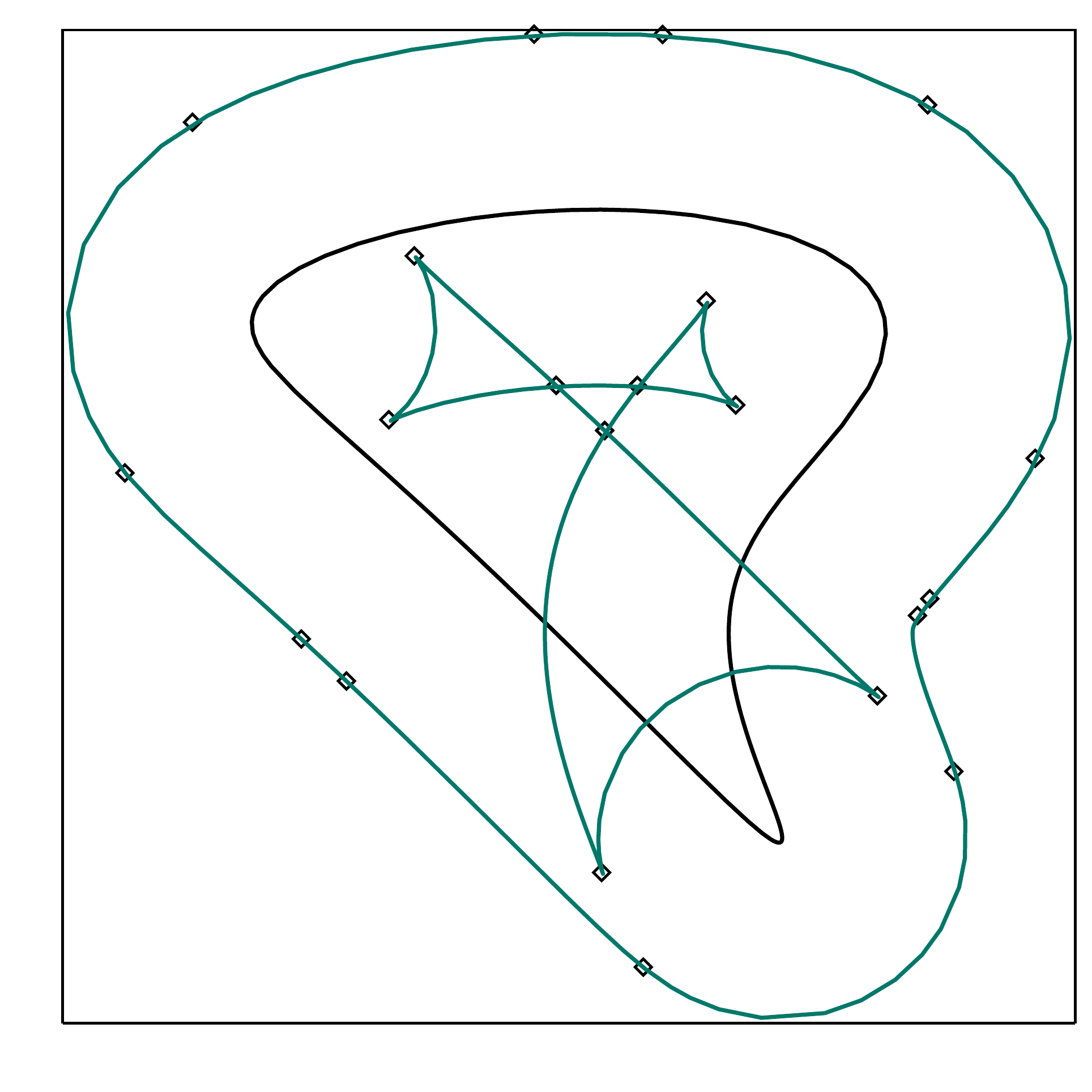} \\
\end{array}$}
\end{center}
\caption{Examples of the algorithm.}\label{figuras_ejemplos}
\end{figure}

\subsection{Comparison with other approaches}
We have also implemented in the same computer algebra system the algorithm described in \cite{Far2}, in order to compare timings. We provide the details of the experiments carried out with several polynomial curves, denoted as $\mathbf{C}_i$, $i\in \{9,\ldots,13\}$, enlisted below:  
\begin{eqnarray*}
\mathbf{C}_9& : & x= 87\,{t}^{5}+44\,{t}^{4}+29\,{t}^{3}+98\,{t}^{2}-23\,t +10,\\&& y=-61\,{t}^{5}-8\,{t}^{4}-29\,{t}^{3}+95\,{t}^{2}+11\,t-49.\\
\\%
\mathbf{C}_{10} & : &x=1+{t}^{4}+2\,{t}^{3}+9\,{t}^{2}-3\,t , \\&& y=-4+{t}^{3}+5\,{t}^{2}+t \\
\\%
\mathbf{C}_{11}& : & x=95\,{t}^{5}+11\,{t}^{4}-49\,{t}^{3}-47\,{t}^{2}+40\,t-81,\\&&  y=98\,{t}^{5}-23\,{t}^{4}+10\,{t}^{3}-61\,{t}^{2}-8\,t-29. \\ 
\\%
\mathbf{C}_{12} &: & x= t,\\&&y=t^4.\\
\\%
\mathbf{C}_{13}& : & x=\frac{-1}{10}\, \left( 1-t \right) ^{6}+\frac{9}{5} \, \left( 1-t \right) ^{5}t-15\,
 \left( 1-t \right) ^{4}{t}^{2}+15\, \left( 1-t \right) ^{2}{t}^{4}-\frac{9}{5}\, \left( 1-t \right) {t}^{5}+\frac{1}{10}\,{t}^{6}
,\\&& y=\left( 1-t \right) ^{6}+{\frac {21}{5}}\, \left( 1-t \right) ^{5}t+9
\, \left( 1-t \right) ^{4}{t}^{2}+9\, \left( 1-t \right) ^{2}{t}^{4}+{
\frac {21}{5}}\, \left( 1-t \right) {t}^{5}+{t}^{6}.
\\
\\
\\ %
\end{eqnarray*}
The following table, Table 2, includes the parameters in Table 1 plus Time$_2$, which is the computing time of the algorithm in \cite{Far2}.  

\begin{center}
\begin{tabular}{|c|l|l|c|c|c|c|c|c|} \hline
Ex. & Time & Time$_2$ & $d$ & $n_p$  & $\delta_t$ &$\tau$ &$\delta_t(P), \delta_t(Q) $   &$\delta(F(x,y)) $ \\
\hline 
\hline $\mathbf{C}_9$ &   0.830 &  60.286 & 2 & 8 &   200 &   660   & 9,10& 	18 \\ 
\hline $\mathbf{C}_{10}$ &  0.247   &4.527 & 1& 4 &   108   &   177   & 7, 8 & 14 \\
\hline $\mathbf{C}_{11}$ &  1.288  &60.443  &5/3 &4&   200 &   696      &9, 10  & 18\\
\hline $\mathbf{C}_{12}$ & 0.107   & 0.667 & 0.8 & 8  &  84 &  99 &7, 8 &  14  \\
\hline $\mathbf{C}_{13}$ & 4.672   &475.200  & 0.05   & 4 &  320&  948  & 11, 12 & 22 \\
\hline $\mathbf{C}_{13}$ &  6.779  & 555.871 &  $\simeq$ 0.03141 & 4 &  320  & 1547&11, 12 & 22 \\
\hline
\end{tabular}
\end{center}\begin{center}
{\bf Table 2:}  Comparative examples
\end{center}

The timings in Table 2 show that our algorithm is clearly much faster than the algorithm in \cite{Far2}. The pictures corresponding to the examples in Table 2 can be found in Figure \ref{figuras_ejemplos2}; from left to right, the curves $\mathbf{C}_9$, $\mathbf{C}_{10}$, $\mathbf{C}_{11}$ appear in the first row, $\mathbf{C}_{12}$, $\mathbf{C}_{13}$ with $d=0.05$ and $\mathbf{C}_{13}$ with $d=\simeq0.3141$ are shown in the second row. The curves  $\mathbf{C}_{12}$ and $\mathbf{C}_{13}$ are also considered in \cite{Maekawa}: $\mathbf{C}_{12}$ is the superbola $y=t^4$ and $\mathbf{C}_{13}$ is the bottle-shaped B\'ezier curve whose control points are given by
$$
(-0.1,1),(0.3,0.7),(-1,0.6),(0,0),(1,0.6),(-0.3,0.7),(0.1,1).
$$
Note that the curve $\mathbf{C}_{13}$ for $d\simeq 0.3141$ is an example of a tacnode. Our algorithm computes four different real $t$-values, but with an appropriate tolerance one observes that only two of these values are regarded as different. In \cite{Maekawa}, several timings are listed for $\mathbf{C}_{12}$ and $\mathbf{C}_{13}$, depending on the implementation method (double precision
floating point arithmetic, or rounded interval arithmetic), and the tolerance used in the computations. Although our timings are better than those in \cite{Maekawa}, here the comparison is less clear, since the method in \cite{Maekawa} is implemented in C++ and the timings correspond to a graphics workstation running at 36 MHz, very different from our own machine; implementing in our own system the technique used in \cite{Maekawa} would be really difficult, since a number of nontrivial algebraic and numerical strategies are involved.
\begin{figure}
\vspace{-2cm}
\begin{center}
\centerline{$\begin{array}{ccc}
\includegraphics[scale=0.25]{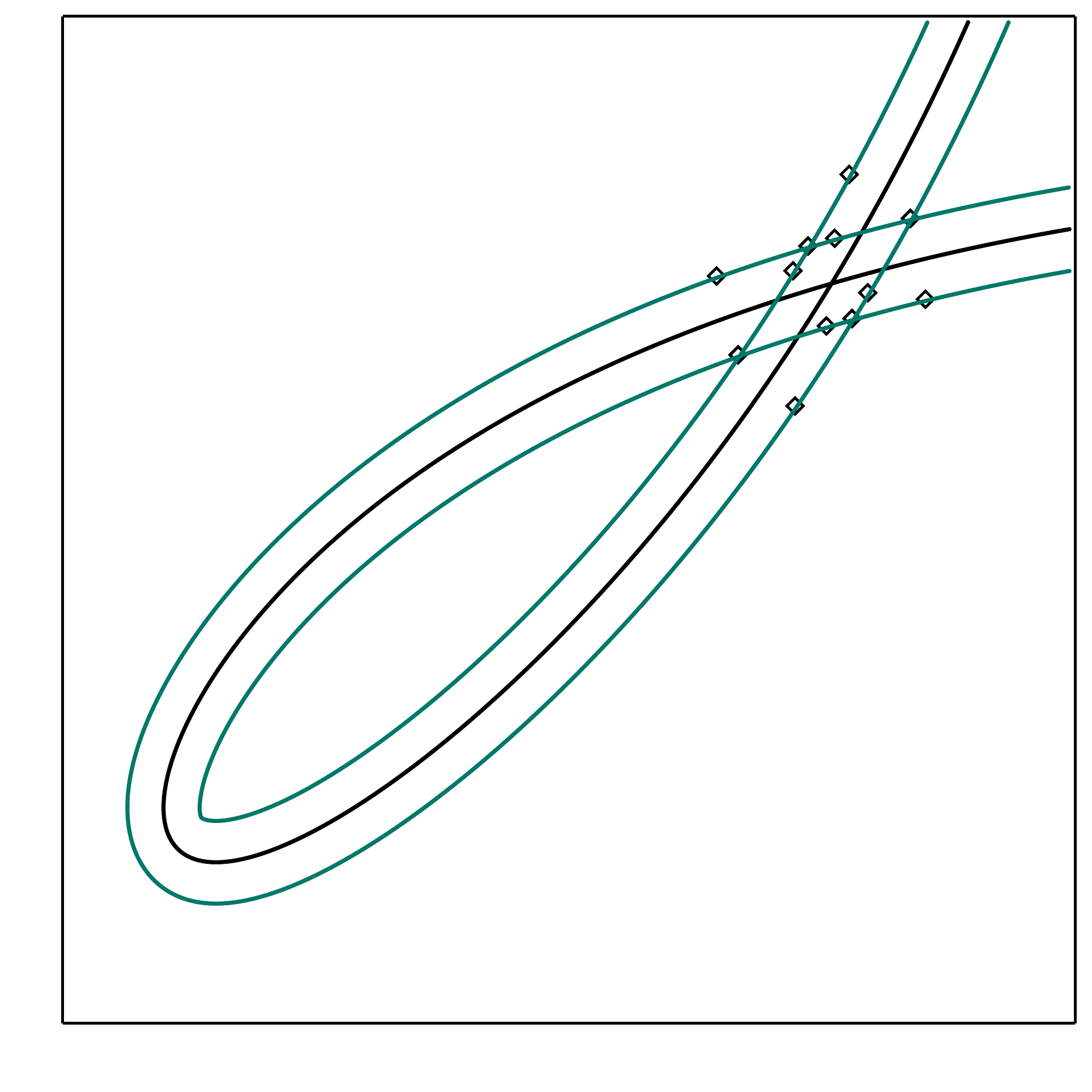}
 &\includegraphics[scale=0.25]{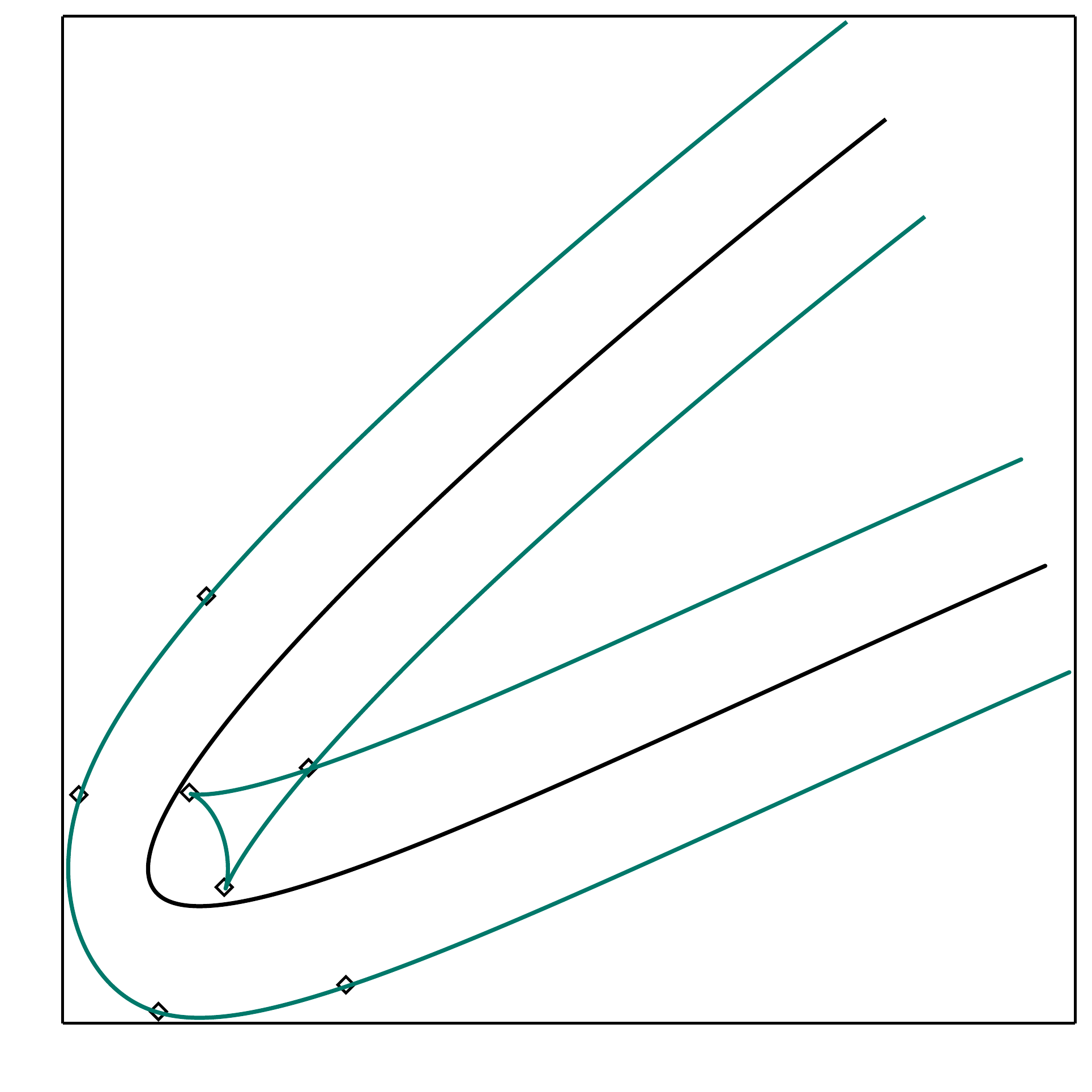}
 &\includegraphics[scale=0.25]{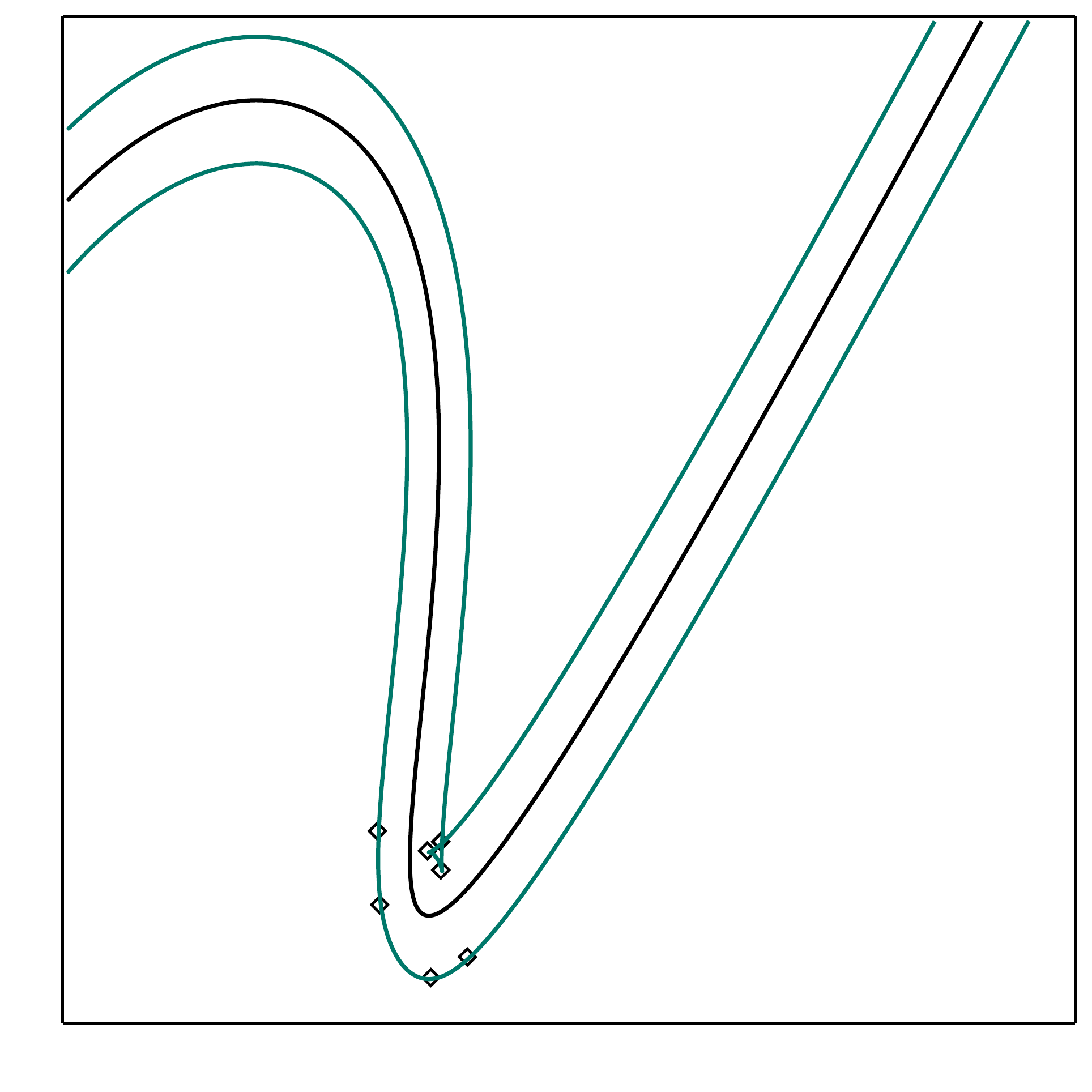} \\
\includegraphics[scale=0.25]{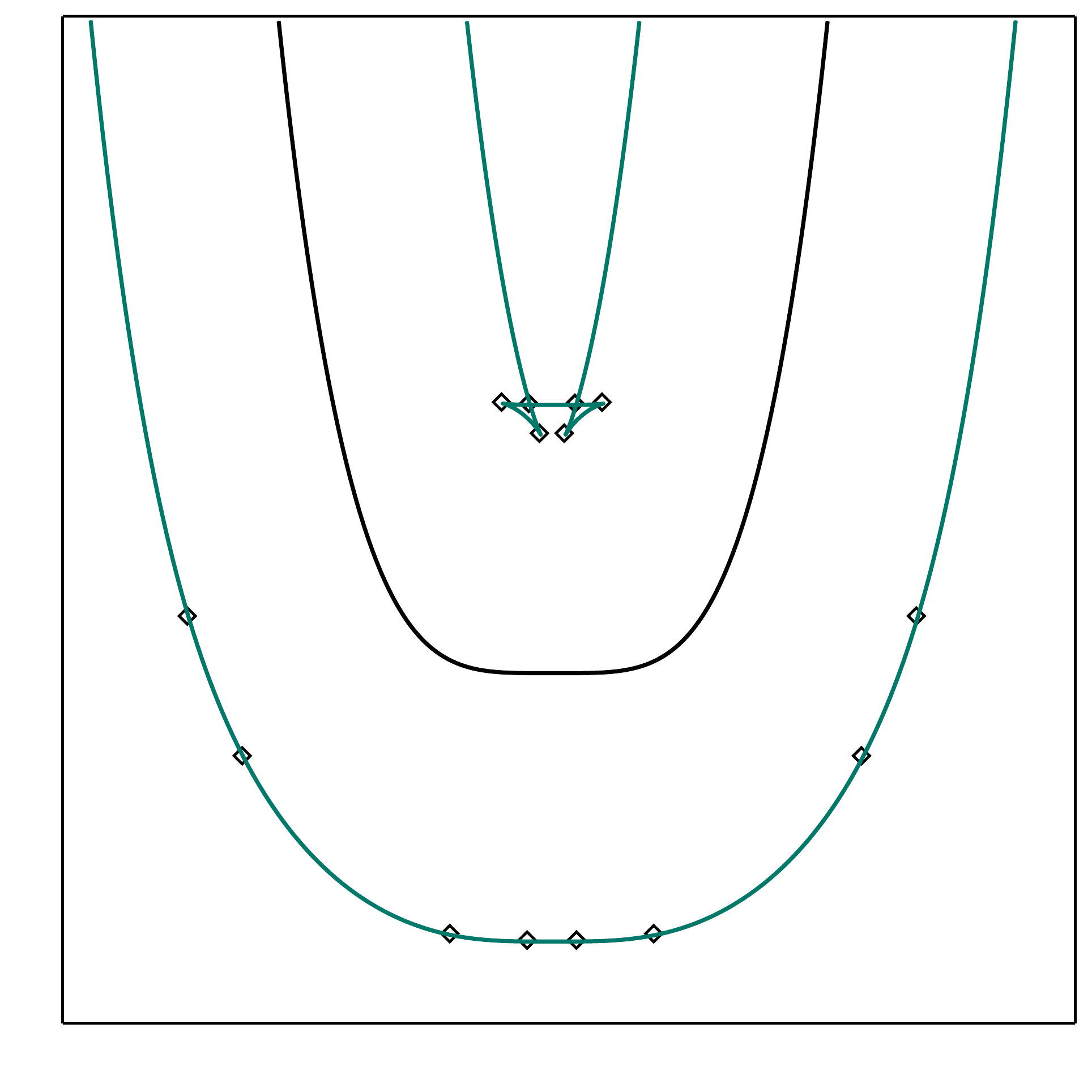} &
\includegraphics[scale=0.25]{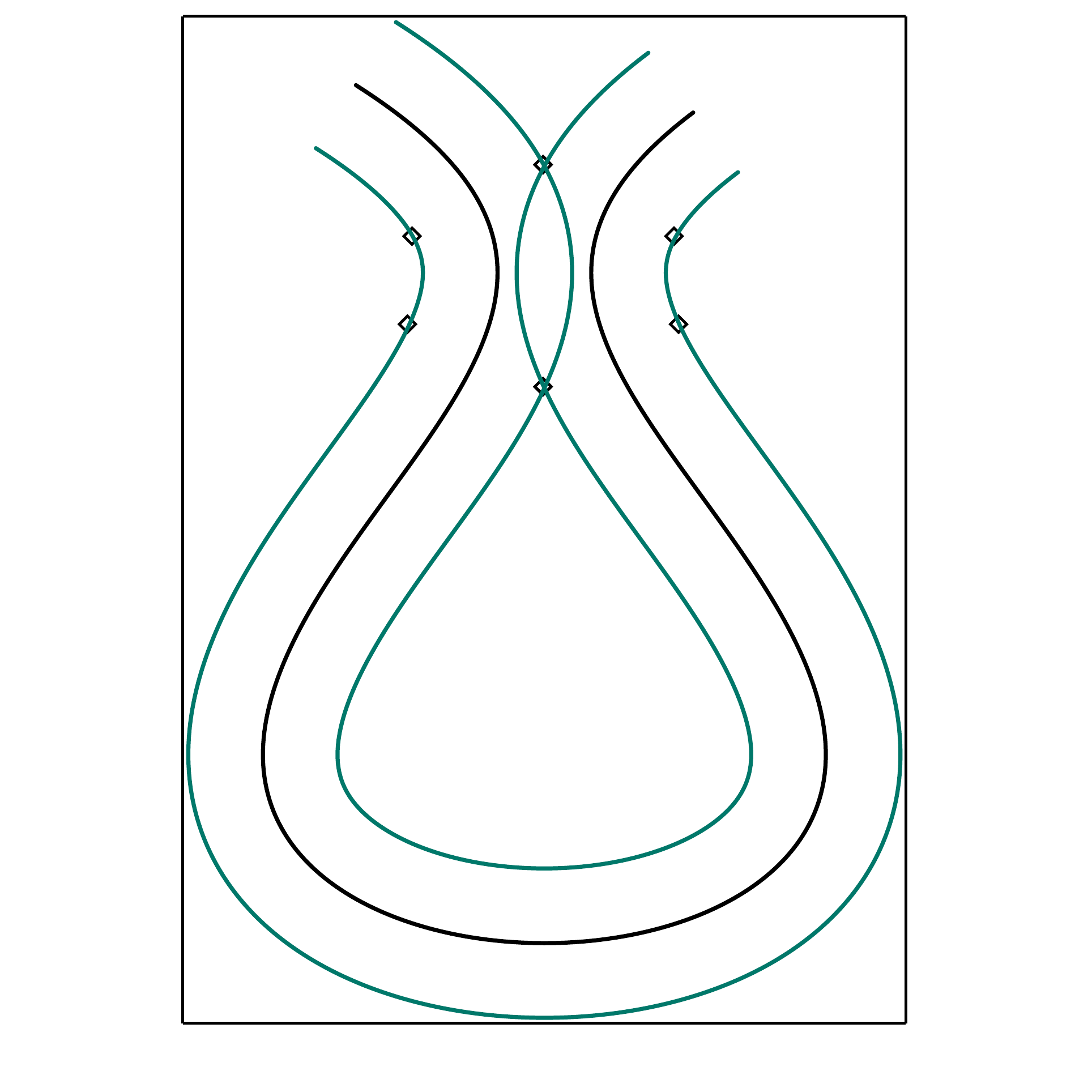} &
\includegraphics[scale=0.25]{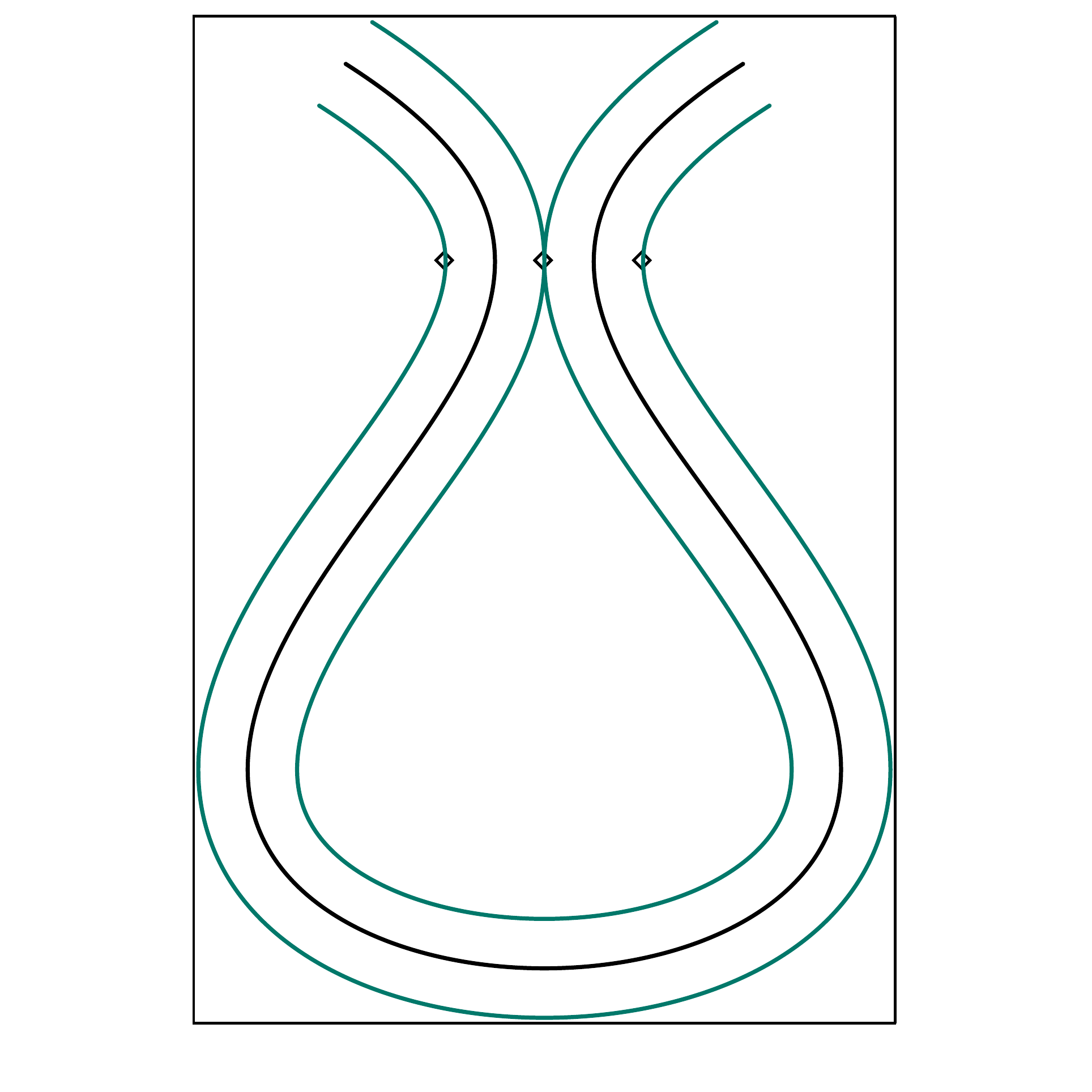}\\
\end{array}$}
\end{center}
\caption{Polynomial examples}\label{figuras_ejemplos2}
\end{figure}

\section{Conclusions.} \label{sec-conclusions}

In this paper we have presented a novel method to compute the real affine, non-isolated singularities of the offset of a planar curve described by means of a proper rational parametrization $({\mathcal X}(t),{\mathcal Y}(t))$. The method is based on ideas in \cite{Fukushima}, and requires the computation of the inverse of a mapping that relates the offset with an auxiliary, simpler curve ${\mathcal M}$. This curve ${\mathcal M}$ lives in the plane $(t,\alpha)$ and has the equation $\alpha^2=\widehat{U}^2(t)+\widehat{V}^2(t)$, where $t$ is the original parameter of the curve and $\alpha$ is an auxiliary variable. The method is easy to describe and to implement, and provides a finite list containing the affine $t$-values generating the real, non-isolated singularities of the offset, which is useful, in particular, for trimming applications. The method can be generalized to other geometric constructions involving square roots studied in the CAGD literature, like for instance \emph{generalized offsets} (\cite{A12}, \cite{Rafa1}) or \emph{conchoids} (\cite{A12-2}, \cite{SS08}, \cite{SS10}).
 

\newpage
\section{Appendix I: proof of Lemma \ref{l-aux}.}

The goal of this appendix is proving Lemma \ref{l-aux}. In order to do it, we need to analyze the behavior of $P, Q$ at the points of ${\mathcal O}_d({\mathcal C})$ generated by $t= \infty$, if any. We introduce the notation $p_{\infty}=\lim\limits_{t\to\infty}\big(\frac{X(t)}{W(t)},\frac{Y(t)}{W(t)}\big)$.

\begin{lemma}\label{circulomalo}
The leading coefficient of $Q$ with respect to $t$ vanishes if and only if $p_{\infty}$ is an affine point. Furthermore, the only points where the leading coefficient of $Q$ with respect to $t$ vanishes are the points of the circle of radius $d$ centered at $p_{\infty}$.
\end{lemma}

\begin{proof}
Note first that the leading coefficient of $Q$ with respect to $t$ is equal to the leading coefficient of $\tilde{Q}$ up to a multiplication by a constant. Dividing (\ref{Q}) by $W^2(t)$, we get
\[f(x,y,t):=\left(x-\frac{X(t)}{W(t)}\right)^2+\left(y-\frac{Y(t)}{W(t)}\right)^2-d^2.\]
One can see that $\deg(W)\ge\max\{\deg(X),\deg(Y)\}$ iff $p_{\infty}$ is affine. Furthermore, if $\deg(W)\ge\max\{\deg(X),\deg(Y)\}$ then $\deg(\tilde{Q})=2\cdot\deg(W)$, and the leading coefficient of $\tilde{Q}$ with respect to $t$ is the product of the leading coefficient of $W^2(t)$, multiplied by $\lim\limits_{t\to \infty}f(x,y,t)$. But this coincides with the equation of the circle of radius $d$ centered at $P_{\infty}$. 

Finally, if $\deg(W)<\max\{\deg(X),\deg(Y)\}$, in which case $p_{\infty}$ is not affine, then the leading coefficient of $\tilde{Q}$ with respect to $t$ is a nonzero constant. 
\end{proof}

\begin{lemma}\label{rectamala}
The leading coefficient of $P$ with respect to $t$ vanishes iff $p_{\infty}$ is an affine point. Furthermore, the leading coefficient of $P$ with respect to $t$ vanishes over the line normal to $\mathcal{C}$ at $p_{\infty}$. 
\end{lemma}

\begin{proof}
First, note that the leading coefficient of $P$ with respect to $t$ is equal to the leading coefficient of $\tilde{P}$, up to a multiplication by a constant. 
Let $r:=\max\{\deg(X),\deg(Y)\}$,  $s:=\max\{\deg(U),\deg(V)\}$ and $w:=\deg(W)$. Furthermore, let $c(W,r)$ and $c(W,w)$ denote the coefficient of $t^r$ and $t^w$ respectively in $W(t)$; similarly for $X$ and $Y$. Also, let $c(U,s)$ denote the coefficient of $t^s$ in $U(t)$; similarly for $c(V,s)$. 

Now $p_{\infty}$ is affine iff $\deg(W)\ge r$. Furthermore, if $\deg(W)\ge r$ then the leading coefficient of $\tilde{P}$ with respect to $t$, is:
\[
\mathrm{lcoeff}(\tilde{P})=c(U,s)\cdot \left[c(W,w)\cdot x-c(X,w)\right]+c(V,s)\cdot \left[c(W,w)\cdot y-c(Y,w)\right].
\]
This expression can be written as the product of the constant $c(W,w)\cdot\sqrt{C^2(U,s)+C^2(V,s)}$ times the dot product of the following two vectors:
\[v_1=\left(\frac{c(U,s)}{\sqrt{c^2(U,s)+c^2(V,s)}},\frac{c(V,s)}{\sqrt{c^2(U,s)+c^2(V,s)}}\right),\]and 
\[
v_2=(x,y)-\left(\frac{c(X,w)}{c(W,w)},\frac{c(Y,w)}{c(W,w)}\right).
\]
Furthermore, one can see that 
\[
v_1=\displaystyle\lim_{t\to\infty}\left(\frac{U(t)}{\sqrt{U^2(t)+V^2(t)}},\frac{ V(t) }{ \sqrt{U^2(t)+V^2(t)}}\right),
\]
which is the limit of the unitary vector tangent to $\mathcal{C}$ at $p_{\infty}$. Additionally, denoting $p_{\infty}=(x_{\infty},y_{\infty})$, we also have that $v_2=(x,y)-(x_{\infty},y_{\infty})$, which represents the vector connecting $p_{\infty}$ and a generic point $(x,y)$. Therefore, $v_1\cdot v_2=0$ is the equation of the line normal to $\mathcal{C}$ at  $p_{\infty}$. 

If $\deg(W)< r$, in which case $p_{\infty}$ is not affine, one can easily see that
\[\mathrm{lcoeff}(\tilde{P})=-(r-w)C(W,w)(C(X,r)^2+C(Y,r)^2),\]
which is always nonzero. 
\end{proof}

Then we can finally prove Lemma \ref{l-aux}:

\begin{proof}(of Lemma \ref{l-aux})  From Lemma \ref{circulomalo} and Lemma \ref{rectamala}, we have that the leading coefficients of $P(x,y,t)$, $Q(x,y,t)$ with respect to $t$ simultaneously vanish at the points $(x_0,y_0)\in {\mathcal O}_d({\mathcal C})$ belonging, at the same time, to the circle of radius $d$ centered at $p_{\infty}$, and to the line normal to ${\mathcal C}$ at $p_{\infty}$. But these are exactly the points of ${\mathcal O}_d({\mathcal C})$ generated by $p_{\infty}$.
\end{proof}

\section{Appendix II: singularities of ${\mathcal O}_d({\mathcal C})$ coming from local singularities of ${\mathcal C}$.}

The goal of this appendix is to prove that the singularities of ${\mathcal O}_d({\mathcal C})$ coming from local singularities of ${\mathcal C}$ also satisfy Equation (\ref{eq-fundam}). In this sense, if $(x_0,y_0)\in \mathcal{O}_d(\mathcal{C})$ is generated by $t_0\in {\Bbb R}$, with $(U(t_0),V(t_0))=(0,0)$, and is a self-intersection, then $\mathbf{sres}_1(x(t_0,\alpha_0),y(t_0,\alpha_0))=0$. Observe that if $(x_0,y_0)\neq P_{\pm \infty}$ then  $\deg(G_{x_0,y_0}(t))>1$ and so, $\mathbf{sres}_1(x(t_0,\alpha_0),y(t_0,\alpha_0))=0$; otherwise, $(x_0,y_0)=P_{\pm \infty}$, $\mathbf{sres}_1(x(t_0,\alpha_0),y(t_0,\alpha_0))=0$ by Lemma \ref{l-aux} and Definition \ref{defsyl}.

Hence, in the rest of the appendix we will focus on the case when $(x_0,y_0)$ is not a self-intersection.
  
\begin{lemma}\label{singular}
 Let $(x_0,y_0)\in \mathcal{O}_d(\mathcal{C})$ be generated by $t_0\in {\Bbb R}$, and assume that $(U(t_0),V(t_0))=(0,0)$. Then the multiplicity of $t_0$ as a root of $G_{x_0,y_0}(t)$ is higher than 1 if and only if 
\[
\widehat{U}(t_0)\widehat{V}'(t_0)-\widehat{V}(t_0)\widehat{U}'(t_0)=0.
\]
\end{lemma}
   
\begin{proof} We observe first that since $W(t_0)\neq 0$, $t_0$ is a root of $Q(x_0,y_0,t)$ of multiplicity $k$ iff $t_0$ is also a root of $\tilde{Q}(x_0,y_0,t)$ of the same multiplicity. Similarly, $t_0$ is a root of $P(x_0,y_0,t)$ of multiplicity $\ell$ iff $t_0$ is also a root of the polynomial \[\widehat{P}(x_0,y_0,t)=\frac{\tilde{P}(x,y,t)}{ \gcd(U(t),V(t))}\]of the same multiplicity. Furthermore, one can easily see that since $W(t_0)\neq 0$, if $t_0$ is a root of $P(x_0,y_0,t)$ of multiplicity $\ell$ then $t_0$ is also a root of $\tilde{Q}(x_0,y_0,t)$, and therefore of $Q(x_0,y_0,t)$, of multiplicity $\ell+1$. So it suffices to show that the condition in the statement of the lemma holds if and only if the multiplicity of $t_0$ as a root of $\widehat{P}(x_0,y_0,t)$ is higher than 1. In order to do this, notice that the derivative of 
$$
\widehat{P}(x_0,y_0,t)/W(t)=\widehat{U}(t)\left( x_0-\frac{X(t)}{W(t)} \right) +\widehat{V}(t)\left( y_0-\frac{Y(t)}{W(t)} \right)
$$
is:
\begin{equation}\label{coco}
\widehat{U}'(t)\left( x_0-\frac{X(t)}{W(t)} \right) +\widehat{V}'(t)\left( y_0-\frac{Y(t)}{W(t)} \right)-\frac{\widehat{U}(t)U(t)+\widehat{V}(t)V(t)}{W^2(t)}.
\end{equation}
Then $t=t_0$, where $W(t_0)\neq 0$, is a root of $\widehat{P}(x_0,y_0,t)$ of multiplicity higher than 1 if and only if \eqref{coco} vanishes at $t=t_0$. 
Since $U(t_0)=V(t_0)=0$ by hypothesis, the evaluation of \eqref{coco} at $t=t_0$ is equal to the dot product of the vectors:
\[\vec{a}=(\widehat{U}'(t_0),\widehat{V}'(t_0))\mbox{, } \vec{b}=\left(x_0-\frac{X(t_0)}{W(t_0)},y_0-\frac{Y(t_0)}{W(t_0)}\right).\]
Since $(x_0,y_0)\in\mathcal{O}_d(\mathcal{C})$ is generated by $t=t_0$, we have that $\vec{b}$ is parallel to 
\[\left(-\frac{\widehat{V}(t_0)}{\sqrt{\widehat{U}^2(t_0)+\widehat{V}^2(t_0)}},\frac{\widehat{U}(t_0)}{\sqrt{\widehat{U}^2(t_0)+\widehat{V}^2(t_0)}}\right).\]
Thus, \eqref{coco} vanishes at $t=t_0$ if and only if
\[\widehat{U}(t_0)\widehat{V}'(t_0)-\widehat{V}(t_0)\widehat{U}'(t_0)=0.\]
\end{proof}
The following corollary follows from Lemma \ref{singular}.
\begin{corollary}\label{cor-prev}
 Under the hypotheses of Lemma \ref{singular}, $\widehat{U}(t_0)\widehat{V}'(t_0)-\widehat{V}(t_0)\widehat{U}'(t_0)=0$ implies $\mathbf{sres}_1(x_0,y_0)=0$.
\end{corollary}

\begin{proof} By Lemma \ref{singular} we have that $\deg(G_{x_0,y_0}(t))>1$. Then one argues as in the proof of Proposition \ref{mainpropo}.
\end{proof}

We need some previous work in order to see the geometrical meaning of Lemma \ref{singular}. Let $T:{\Bbb R}^2\to {\Bbb R}^2$ be an orthogonal change of coordinates
\[T(x,y)=(ax+by+c_1,-bx+ay+c_2),\]
with $a^2+b^2=1.$ Let also 
\[
\tilde{\phi}_T(t)=T({\mathcal X}(t),{\mathcal Y}(t))=(\tilde{\mathcal X}(t),\tilde{\mathcal Y}(t)),
\]
with  $\tilde{\mathcal X}(t)=a{\mathcal X}(t)+b{\mathcal Y}(t)+c_1$, $\tilde{\mathcal Y}(t)=-b{\mathcal X}(t)+a{\mathcal Y}(t)+c_2$, and let 
\[
 T(U(t),V(t))=(\tilde{U}(t),\tilde{V}(t))\text{ and }T(\widehat{U}(t),\widehat{V}(t))=(\widehat{U}_T(t),\widehat{V}_T(t)).
\]
Hence, it is easy to see that 
\[
\tilde{\mathcal X}'(t)=\frac{\tilde{U}(t)}{W^2(t)}, \,\,\,
\tilde{\mathcal Y}'(t)=\frac{\tilde{V}(t)}{W^2(t)}. 
\]  Moreover, since
\[
\gcd(\tilde{U},\tilde{V})=\gcd(aU(t)+bV(t),-bU(t)+aV(t))=\gcd(U(t),V(t)),
\]
then 
\[\widehat{U}_T(t)=\tilde{U}(t)/\gcd(\tilde{U}(t),\tilde{V}(t)), \mbox{ }\widehat{V}_T(t)=\tilde{V}(t)/\gcd(\tilde{U}(t),\tilde{V}(t)).\]
Now we have the following instrumental lemma.
\begin{lemma} \label{orth}
Let \[\xi(t)={\mathcal X}'(t){\mathcal Y}''(t)-{\mathcal X}''(t){\mathcal Y}'(t),\]and let 
\[\tilde{\xi}_T(t)=\tilde{\mathcal X}'(t)\tilde{\mathcal Y}''(t)-\tilde{\mathcal X}''(t)\tilde{\mathcal Y}'(t).\]
The following statements are true:
\begin{itemize}
\item [(1)] $\xi(t)=\tilde{\xi}_T(t)$.
\item [(2)] $\widehat{U}_T(t)\widehat{V}_T'(t)-\widehat{U}_T'(t)\widehat{V}_T(t)=\widehat{U}(t)\widehat{V}'(t)-\widehat{U}'(t)\widehat{V}(t)$
\end{itemize}
\end{lemma}

\begin{proof}
(1) The equality can be verified by a direct computation. (2) Let $\nu(t)=\gcd(U(t),V(t))$. Observe first that 
\[{\mathcal X}'(t)=\frac{U(t)}{W^2(t)}=\frac{\nu(t)\widehat{U}(t)}{W^2(t)},\mbox{ }{\mathcal Y}'(t)=\frac{V(t)}{W^2(t)}=\frac{\nu(t)\widehat{V}(t)}{W^2(t)},\]
and therefore
\[\xi(t)=\frac{\nu^2(t)\cdot (\widehat{U}(t)\widehat{V}'(t)-\widehat{U}'(t)\widehat{V}(t))}{W^4(t)}.\]
Furthermore, 
\[\tilde{\mathcal X}'(t)=\frac{\tilde{U}(t)}{W^2(t)}=\frac{\nu(t)\widehat{U}_T(t)}{W^2(t)},\mbox{ }\tilde{\mathcal Y}'(t)=\frac{\tilde{V}(t)}{W^2(t)}=\frac{\nu(t)\widehat{V}_T(t)}{W^2(t)},\]
and hence
\[\tilde{\xi}_T(t)=\frac{\nu^2(t)\cdot (\widehat{U}_T(t)\widehat{V}_T'(t)-\widehat{U}_T'(t)\widehat{V}_T(t))}{W^4(t)}.\]
Finally, using the statement (1), the statement (2) follows. 
\end{proof}
 
Before giving a geometric translation of Lemma \ref{singular}, we need an additional ingredient, namely the notion of \emph{place} recalled in Subsection \ref{sec-sing}. So let $S\in {\mathcal C}$ be a local, real affine singularity of ${\mathcal C}$, $S=\phi(t_0)$, $t_0\in {\Bbb R}$, and let us consider a coordinate system centered at $S$ where the $x$-axis coincides with the tangent to ${\mathcal D}$ at $S$ (see Figure \ref{fig:orth}). Then $S$ is the center of a place of ${\mathcal C}$ that can be written as 
\[{\mathcal P}(h)=(h^p,\beta_qh^q+\cdots),\]with $p,q\in {\Bbb N}$, $p\geq 2$, $q>p$. In \cite{AS07}, the problem of determining the places ${\mathcal P}_{\pm d}(h)$ of ${\mathcal O}_d({\mathcal C})$ generated by a given place ${\mathcal P}(h)$ of ${\mathcal C}$ under the offsetting transformation is addressed. Furthermore, from Theorem 7 in \cite{AS07}, one can see that if $S$ is a local singularity of ${\mathcal C}$, then it generates a local singularity of ${\mathcal O}_d({\mathcal C})$ if and only if $q-p\geq 2$. Now let us see that Lemma \ref{singular} implies that $S$ must generate a singularity of $\mathcal{O}_d(\mathcal{C})$. This is done in the following proposition, where we use the previous notation and the ideas in Lemma \ref{orth}.

\begin{figure}
\begin{center}
\includegraphics[scale=0.5]{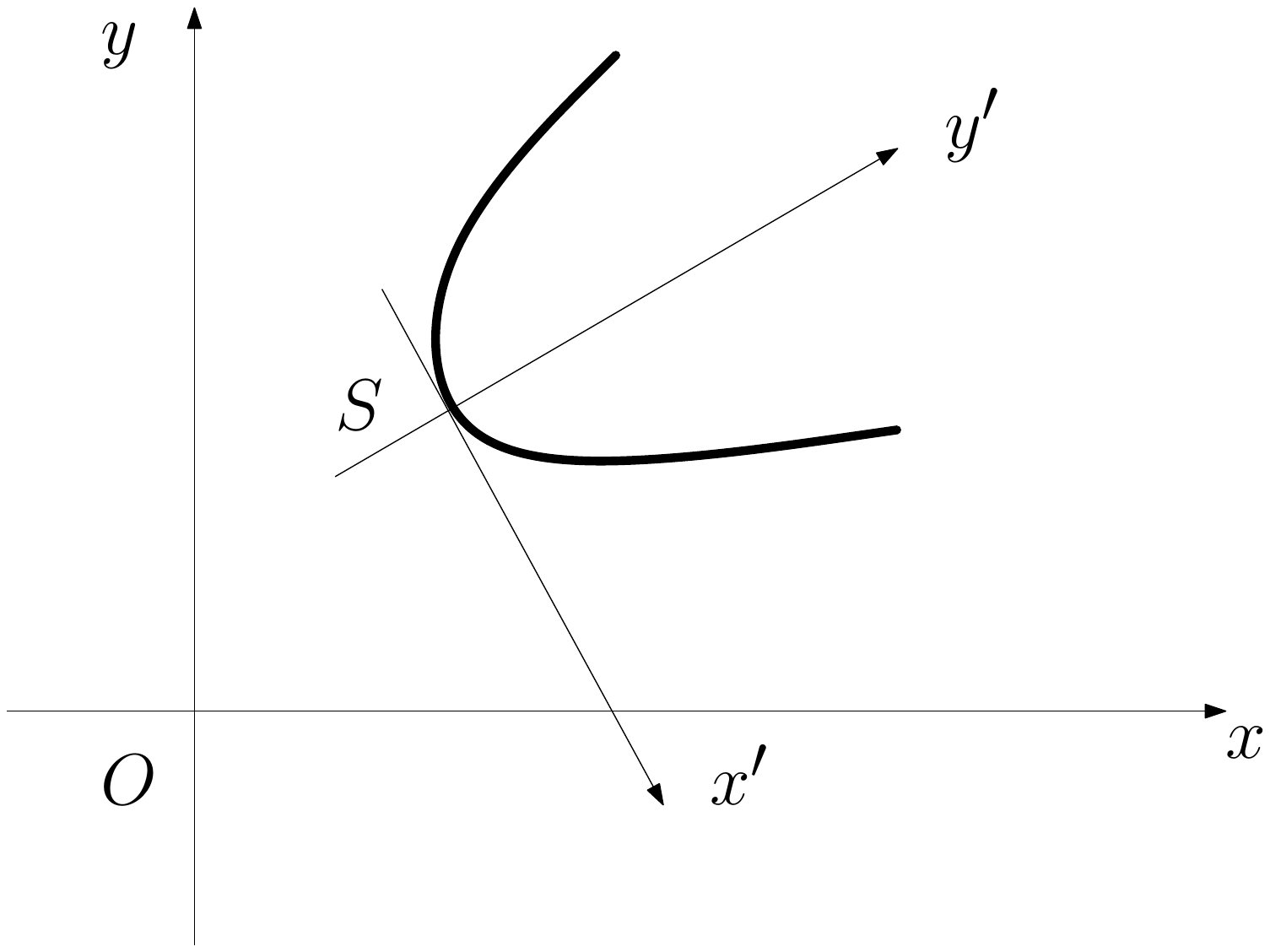}
\end{center}
\caption{Coordinate system for Proposition \ref{geom-mean}.}\label{fig:orth}
\end{figure}

\begin{proposition} \label{geom-mean}
Let $(x_0,y_0)\in \mathcal{O}_d(\mathcal{C})$ be generated by $t_0\in {\Bbb R}$, and assume that $(U(t_0),V(t_0))=(0,0)$. Then $(x_0,y_0)$ is a local singularity of $\mathcal{O}_d(\mathcal{C})$ iff $\widehat{U}(t_0)\widehat{V}'(t_0)-\widehat{V}(t_0)\widehat{U}'(t_0)=0.$
\end{proposition}

\begin{proof} $(\Leftarrow)$ Let $S=\phi(t_0)\in {\mathcal C}$, and let $x'$ be the tangent line to ${\mathcal C}$ at $S$ corresponding to $t\to t_0$. Furthermore, let $y'$ be the line perpendicular to $x'$ at $S$, and let $T$ be the orthogonal change of coordinates mapping the coordinate system $\{O;x,y\}$ onto the coordinate system $\{S;x',y'\}$ (see Figure \ref{fig:orth}). Also, let $(\tilde{\mathcal X}(t),\tilde{\mathcal Y}(t))=T({\mathcal X}(t),{\mathcal Y}(t))$; notice that $(\tilde{\mathcal X}(t),\tilde{\mathcal Y}(t))$ is a parametrization of ${\mathcal C}$ in the coordinate system $\{S;x',y'\}$. 
Since 
\[
(\tilde{\mathcal X}'(t),\tilde{\mathcal Y}'(t))= \frac{\gamma(t)}{W^2(t)} \cdot (\widehat{U}_T(t),\widehat{V}_T(t)),
\]
we have that 
\[
m(t):=\frac{\tilde{\mathcal Y}'(t)}{\tilde{\mathcal X}'(t)}=\frac{\widehat{V}_T(t)}{\widehat{U}_T(t)}.\]
Observe that the tangent to ${\mathcal C}$ at $S$ is parallel to $(\widehat{U}_T(t_0),\widehat{V}_T(t_0))\neq (0,0)$. Since this tangent is parallel to the $x'$-axis, $\widehat{U}_T(t_0)\neq 0$ and $m(t_0)$ is well-defined.  
Furthermore, 
\[
m'(t)=\frac{\widehat{V}_T'(t)\widehat{U}_T(t)-\widehat{U}'_T(t)\widehat{V}_T(t)}{\widehat{U}_T^2(t)},
\]
and by the condition (ii) and the statement (2) of Lemma \ref{orth}, we have that $m'(t_0)=0$. 

On the other hand, let us consider a place $(x(h),y(h))=(h^p,\beta_qh^q+\cdots)$ of ${\mathcal C}$ centered at $S$. Then we have that 
\begin{equation}\label{lan}
n(h):=\frac{y'(h)}{x'(h)}=\frac{q\beta_q h^{q-1}+\cdots}{ph^{p-1}}=\frac{q\beta_q}{p}h^{q-p}+\cdots
\end{equation}
Since $(x(h),y(h))$ also parametrizes ${\mathcal C}$ around $S$, and is written in the same coordinate system as $(\tilde{\mathcal X}(t),\tilde{\mathcal Y}(t))$, namely the coordinate system $\{S;x',y'\}$, for every $t$ sufficiently close to $t_0$ we can find $h$ such that 
\[
\frac{y'(h)}{x'(h)}=\frac{\tilde{\mathcal Y}'(t)}{\tilde{\mathcal X}'(t)}.
\]
Therefore $n'(0)=m'(t_0)=0$. However, $n'(0)=0$ implies $q-p\geq 2$, which is the condition (see Theorem 7 in \cite{AS07}) for $S$ to generate a singular point in $\mathcal{O}_d(\mathcal{C})$ .

$(\Rightarrow)$ If $(x_0,y_0)$ is a local singularity then $q-p\geq 2$, and therefore $n'(0)=m'(t_0)=0$; but then $\widehat{U}(t_0)\widehat{V}'(t_0)-\widehat{V}(t_0)\widehat{U}'(t_0)=0$.
\end{proof}
Finally, we can eventually prove the aimed result. 
\begin{corollary}\label{corol-wanted}
If $(x_0,y_0)\in {\mathcal O}_d({\mathcal C})$ is a non-isolated, real affine singularity of $\mathcal{O}_d(\mathcal{C})$, generated by $t_0\in {\Bbb R}$, with $(U(t_0),V(t_0))=(0,0)$, then there exists $\alpha_0\in {\Bbb R}$, $\alpha^2_0=\widehat{U}^2(t_0)+\widehat{V}^2(t_0)$, such that $(x_0,y_0)=(x(t_0,\alpha_0),y(t_0,\alpha_0))$ and $\mathbf{sres}_1(x(t_0,\alpha_0),y(t_0,\alpha_0))=0$.
\end{corollary}

\begin{proof} If $(x_0,y_0)$ is a self-intersection of $\mathcal{O}_d(\mathcal{C})$, the result follows from the observations at the beginning of the appendix. If $(x_0,y_0)$ is a local singularity of $\mathcal{O}_d(\mathcal{C})$, then the result follows from Corollary \ref{cor-prev} and Proposition \ref{geom-mean}.
\end{proof}

\newpage
\section{Appendix III: structure of $\mbox{Res}_t(P(x,y,t),Q(x,y,t))$.}

From Section \ref{subsec-off} we know that the polynomial \[H(x,y)=\mbox{Res}_t(P(x,y,t),Q(x,y,t))\] can be written as $H(x,y)=F(x,y)\cdot G(x,y)$, where $F(x,y)=(f(x,y))^r$, $f(x,y)$ is an irreducible polynomial implicitly representing ${\mathcal O}_d({\mathcal C})$\footnote{Recall that since ${\mathcal O}_d({\mathcal C})$ is irreducible by hypothesis, then its implicit equation consists of just one irreducible factor}, and $G(x,y)$ is the product of all the extraneous factors. Our goal is to prove that, under the hypothesis that ${\mathcal O}_d({\mathcal C})$ is simple and ${\mathcal C}$ is properly parametrized, we have $r=1$, i.e. the component of $H(x,y)$ corresponding to ${\mathcal O}_d({\mathcal C})$ has multiplicity 1. 

Let us consider the set ${\mathcal A}$ of the $y_0$s satisfying some of the following conditions. Here we use the notation $P_{\pm \infty}=(x_{\pm \infty},y_{\pm \infty})$.

\begin{itemize}
\item [(1)] The intersection of the line $y=y_0$ with ${\mathcal O}_d({\mathcal C})$ contains some point also belonging to the curve $G(x,y)=0$.
\item [(2)] The leading coefficients of $P(x,y,t)$ and $Q(x,y,t)$ with respect to $t$ identically vanish when $y=y_0$. 
\item [(3)] $y_0= y_{\pm \infty}$.
\item [(4)] The line $y=y_0$ is tangent to ${\mathcal O}_d({\mathcal C})$.
\item [(5)] The line $y=y_0$ contains either a local singularity of ${\mathcal O}_d({\mathcal C})$, or a point of ${\mathcal O}_d({\mathcal C})$ simultaneously generated by different values\footnote{Notice that these last points correspond to self-intersections of ${\mathcal O}_d({\mathcal C})$} of $t$, i.e. by different points $p=\phi(t)$. 
\end{itemize}

Then we have the following result.

\begin{lemma} \label{finite}
${\mathcal A}$ is a finite set.
\end{lemma}

\begin{proof} It is clear that there are finitely many $y_0$s satisfying (1), (2), (3) and (4). So let us see that there are also finitely many $y_0$s satisfying (5). For this purpose, note that ${\mathcal O}_d({\mathcal C})$ has finitely many local singularities. Furthermore, since ${\mathcal O}_d({\mathcal C})$ is simple by hypothesis and ${\mathcal C}$ is properly parametrized, then there are finitely many points of ${\mathcal O}_d({\mathcal C})$ generated by different values of the paramemeter $t$. Therefore, there are also finitely many $y_0$s satisfying (5).
\end{proof}

Therefore, a \emph{generic} $y_0$ does not satify any condition (1)-(5). This is crucial in the next result.

\begin{proposition} \label{libre}
Let $H(x,y)=F(x,y)\cdot G(x,y)$. Then $F(x,y)$ is a square-free polynomial.
\end{proposition}

\begin{proof} Since $F(x,y)=(f(x,y))^r$, we must prove that $r=1$. In order to do this, observe first that $f(x,y)$ cannot be of the form $\alpha\cdot y$, with $\alpha\in {\Bbb R}$. Indeed, in that case ${\mathcal C}$ consists of a pair of parallel lines, which does not correspond to a rational curve. Therefore in order to prove the assertion, it suffices to prove that for a generic $y_0$, $F(x,y_0)$ is square-free. For this purpose let $y=y_0$ be generic; hence, $y_0$ does not satisfy any condition (1)-(5). In particular, the leading coefficients of $P(x,y,t)$ and $Q(x,y,t)$ do not identically vanish for $y=y_0$, so $\mbox{Res}_t(P(x,y,t),Q(x,y,t))$ specializes properly, i.e. $H(x,y_0)=\mbox{Res}_t(P(x,y_0,t),Q(x,y_0,t))$ (see Theorem \ref{th-2} in Section \ref{sec-subres}). Since $y_0$ does not satisfy condition (1), the line $y=y_0$ does not intersect $H(x,y)=0$ in any point both belonging to ${\mathcal O}_d({\mathcal C})$ and to the curve $G(x,y)=0$. Moreover $y_0$ does not satisfy condition (4) either, and therefore we have that 
\[H(x,y_0)=(x-x_1)^r\cdots (x-x_n)^r\cdot G(x,y_0),\]where $G(x_i,y_0)\neq 0$ for $i=1,\ldots,n$. In other words, the intersection of the line $y=y_0$ with the curve $F(x,y)=0$ consists of the points $(x_i,y_0)$, $i=1,\ldots,n$, and all these points have the same multiplicity of intersection with $y=y_0$, namely $r$. Now assume that $r>1$. From Proposition 5 in \cite{Buse}, we have the following possibilities:
\begin{itemize}
\item [(i)] For any $i=1,\ldots,n$ there is $y_1\neq y_0$ and $t_0,t_1$ with $P(x_i,y_0,t_0)=Q(x_i,y_0,t_0)=0$ and $P(x_i,y_0,t_1)=Q(x_i,y_0,t_1)=0$. But this implies that $(x_i,y_0)\in{\mathcal O}_d({\mathcal C})$ is simultaneously generated by $t_0$ and $t_1$, which cannot happen because $y_0$ does not satisfy condition (5). 
\item [(ii)] There is some $i=1,\ldots,n$ such that $P(x_i,y_0,t)$ and $Q(x_i,y_0,t)$ share a root $t_0$ of multiplicity at least 2. However, in that case $P_t$ vanishes at the point $(x_i,y_0,t_0)$, which by Lemma \ref{lem-1} implies that $(x_i,y_0)$ is a local singularity of ${\mathcal O}_d({\mathcal C})$. But this cannot happen because $y_0$ does not satisfy condition (5).  
\item [(iii)] The line $x=x_i$ is a common vertical asymptote of the curves (defined on the $xt$-plane) $P(x,y_0,t)=0$, $Q(x,y_0,t)=0$. But this cannot happen because $y_0$ does not satisfy condition (3), and hence $y_0\neq y_{\pm \infty}$.
\end{itemize}
So we conclude that $r=1$. 
\end{proof}


\end{document}